\documentclass[12pt]{amsart}
\usepackage{graphicx,amscd,color,amsmath,amsfonts,amssymb}
\usepackage{geometry}
\newtheorem{theorem}{Theorem}[section]
\newtheorem{remark}[theorem]{Remark}
\newtheorem{lemma}[theorem]{Lemma}

\newtheorem{proposition}[theorem]{Proposition}
\newtheorem{corollary}[theorem]{Corollary}

\newtheorem{conjecture}[theorem]{Conjecture}

\newtheorem{problem}[theorem]{Problem}

\numberwithin{equation}{section}
\input epsf
\begin{document}
\title[Bohnenblust--Hille constants]{There exist multilinear Bohnenblust--Hille constants $\left(  C_{n}\right)
_{n=1}^{\infty}$ with $\displaystyle \lim_{n\rightarrow \infty}\left(  C_{n+1}-C_{n}\right)  =0$}
\author[D. Nu\~{n}ez et al.]{D. Nu\~{n}ez-Alarc\'{o}n\textsuperscript{*}, D. Pellegrino\textsuperscript{**} \and J.B. Seoane-Sep\'{u}lveda\textsuperscript{***} \and D. M. Serrano-Rodr\'{\i}guez\textsuperscript{*}}
\address{Departamento de Matem\'{a}tica,\newline\indent Universidade Federal da Para\'{\i}ba,\newline\indent 58.051-900 - Jo\~{a}o Pessoa, Brazil.}
\email{danielnunezal@gmail.com}
\thanks{\textsuperscript{*}Supported by Capes.}
\address{Departamento de Matem\'{a}tica,\newline\indent Universidade Federal da Para\'{\i}ba,\newline\indent 58.051-900 - Jo\~{a}o Pessoa, Brazil.}
\email{dmpellegrino@gmail.com and pellegrino@pq.cnpq.br}
\thanks{\textsuperscript{**}Supported by CNPq Grant 301237/2009-3.}
\address{Departamento de An\'{a}lisis Matem\'{a}tico,\newline\indent Facultad de Ciencias Matem\'{a}ticas, \newline\indent Plaza de Ciencias 3, \newline\indent Universidad Complutense de Madrid,\newline\indent Madrid, 28040, Spain.}
\email{jseoane@mat.ucm.es}
\thanks{\textsuperscript{***}Supported by grant MTM2009-07848.}
\address{Departamento de Matem\'{a}tica,\newline\indent Universidade Federal da Para\'{\i}ba,\newline\indent 58.051-900 - Jo\~{a}o Pessoa, Brazil.}
\email{dmserrano0@gmail.com}
\begin{abstract}
The $n$-linear Bohnenblust--Hille inequality asserts that there is a constant $C_{n}\in\lbrack1,\infty)$ such that the $\ell_{\frac{2n}{n+1}}$-norm of $\left(  U(e_{i_{^{1}}},\ldots,e_{i_{n}})\right)  _{i_{1},\ldots, i_{n}=1}^{N}$ is bounded above by $C_{n}$ times the supremum norm of $U,$ regardless of the $n$-linear form $U:\mathbb{C}^{N}\times\cdots\times\mathbb{C}^{N} \rightarrow\mathbb{C}$ and the positive integer $N$ (the same holds for real scalars). The power $2n/(n+1)$ is sharp but the values and asymptotic behavior of the optimal constants remain a mystery. The first estimates for these constants had exponential growth. Very recently, a new panorama emerged and the importance, for many applications, of the knowledge of the optimal constants (denoted by $\left(  K_{n}\right)  _{n=1}^{\infty}$) was stressed. The title of this paper is part of our \emph{Fundamental Lemma}, one of the novelties presented here. It brings new (and precise) information on the optimal constants (for both real and complex scalars). For instance,
\[
K_{n+1}-K_{n}<\frac{0.87}{n^{0.473}}
\]
for infinitely many $n$'s. In the case of complex scalars we present a curious formula, where $\pi,e$ and the famous Euler--Mascheroni constant $\gamma$ appear together:
\[
K_{n}<1+\left(  \frac{4}{\sqrt{\pi}}\left(  1-e^{\gamma/2-1/2}\right){\sum\limits_{j=1}^{n-1}}j^{^{\log_{2}\left(  e^{-\gamma/2+1/2}\right)  -1}
}\right)
\]
for all $n\geq2$. Numerically, the above formula shows a surprising low growth,
\[
K_{n}<1.41\left(  n-1\right)^{0.305}-0.04
\]
for every integer $n\geq2$. We also provide a brief discussion on the interplay between the Kahane--Salem--Zygmund and the Bohnenblust--Hille (polynomial and multilinear) inequalities. We shall adapt some of the techniques presented here to estimate the constants satisfying Bohnenblust--Hille type inequalities when the exponent $\frac{2n}{n+1}$ is replaced by any $q\in\left[\frac{2n}{n+1},\infty\right)$.
\end{abstract}
\subjclass[2010]{46G25, 30B50}
\keywords{Bohnenblust--Hille inequality, Kahane--Salem--Zygmund inequality, Quantum Information Theory}
\maketitle
\section{Introduction}
The polynomial and multilinear Bohnenblust--Hille inequalities have important applications in different fields of Mathematics and Physics, such as Operator Theory, Fourier and Harmonic Analysis, Complex Analysis, Analytic Number Theory and Quantum Information Theory (see \cite{defant, diniz2} and references therein). Since its proof, in the \emph{Annals of Mathematics} in
1931, the (multilinear and polynomial) Bohnenblust--Hille inequalities were overlooked for decades (see \cite{bh}) and only returned to the spotlights in the last few years with works of A. Defant, L. Frerick, J. Ortega-Cerd\'{a}, M. Ouna\"{\i}es, D. Popa, U. Schwarting, K. Seip, among others. The polynomial Bohnenblust--Hille inequality proves the existence of a positive function
$C:\mathbb{N}\rightarrow\lbrack1,\infty)$ such that for every $m$-homogeneous polynomial $P$ on $\mathbb{C}^{N}$, the $\ell_{\frac{2m}{m+1}}$-norm of the set of coefficients of $P$ is bounded above by $C(m)$ times the supremum norm of $P$ on the unit polydisc. The original estimates for $C(m)$ had a growth of order $m^{m/2}$ and only in 2011 (\cite{annals}) the importance of this inequality was rediscovered and the estimates for $C(m)$ were substantially improved; in the aforementioned paper it is proved that $C(m)$ can be chosen to be hypercontractive and, more precisely,
\begin{equation}
C(m)\leq\left(  1+\frac{1}{m}\right)  ^{m-1}\sqrt{m}\left(  \sqrt{2}\right)^{m-1}. \label{defant55}%
\end{equation}
This result, besides its mathematical importance, has striking applications in different contexts (see \cite{annals}). The multilinear version of the Bohnenblust--Hille inequality has a similar, \textit{mutatis mutandis,} formulation:
\medskip
\noindent\textbf{Multilinear Bohnenblust--Hille inequality.} For every
positive integer $m\geq1$ there exists a sequence of positive scalars $\left(
C_{m}\right)  _{m=1}^{\infty}$ in $[1,\infty)$ such that
\[
\left(  \sum\limits_{i_{1},\ldots,i_{m}=1}^{N}\left\vert U(e_{i_{^{1}}}
,\ldots,e_{i_{m}})\right\vert ^{\frac{2m}{m+1}}\right)  ^{\frac{m+1}{2m}}\leq
C_{m}\sup_{z_{1},\ldots,z_{m}\in\mathbb{D}^{N}}\left\vert U(z_{1},\ldots,z_{m}
)\right\vert
\]
for all $m$-linear forms $U:\mathbb{C}^{N}\times\cdots\times\mathbb{C}
^{N}\rightarrow\mathbb{C}$ and every positive integer $N$, where $\left(
e_{i}\right)  _{i=1}^{N}$ denotes the canonical basis of $\mathbb{C}^{N}$ and
$\mathbb{D}^{N}$ represents the open unit polydisk in $\mathbb{C}^{N}$.
\medskip
The case $m=2$ is the well-known Littlewood's $4/3$ theorem (see
\cite{defant2, Garling, Litt}). The original purpose of Littlewood's $4/3$
theorem was to solve a problem of P.J. Daniell on functions of bounded
variation (see \cite{Litt}); on the other hand, the Bohnenblust--Hille
inequality was invented to solve Bohr's famous absolute convergence problem
within the theory of Dirichlet series (this subject is being recently explored
by several authors; see \cite{Blas, bom, dgm, dgm2, dmp, defant3, defant22}
and references therein). Some independent results were proven in the 1970's
where better upper bounds for $C_{m}$ were obtained, but it seems that the
authors were not aware of the existence of the original results by Bohnenblust
and Hille.

The oblivion of the work of Bohnenblust and Hille in the past was so
noticeable that Blei's book \cite{Blei} (2001) states the Bohnenblust--Hille
inequality as \textquotedblleft the Littlewood's $2n/(n+1) $%
-inequality\textquotedblright\ and absolutely no mention to the paper of
Bohnenblust and Hille is made at all. According to Blei's book the
\textquotedblleft Littlewood's $2n/(n+1)$-inequality\textquotedblright\ is
originally due to A.M. Davie (\cite{d}, 1973) and (independently) to G.
Johnson and W. Woodward (\cite{JW}, 1974) but as a matter of fact Bohnenblust
and Hille's paper preceded the aforementioned works in more than 40 years.

The present paper is divided into eleven short Sections, and an Appendix, as follows:

In Section 2 we describe the main advances and uncertainties related to the search of the sharp constants in the multilinear Bohnenblust--Hille inequality; in Section 3 we describe the state-of-the-art of the subject; in Section 4 we introduce some notation and announce the key result of this paper, called \textit{Fundamental Lemma}; in Section 5 we present our main results; Sections 6,7 and 8 are focused on the proof of technical lemmata, the Fundamental Lemma and the main results; in Section 9 we sketch the same results in the case of complex scalars.

Section 10 is divided in three subsections devoted to the interplay between the Kahane--Salem--Zygmund inequality and the Bohnenblust--Hille inequalities. The first subsection contains a quite straightforward proof of the optimality of the power $2m/(m+1)$ in the polynomial and multilinear Bohnenblust--Hille inequalities. This result is well-known but, to the best of the authors
knowledge, there is no simple proof of this fact in the literature. According to Defant \textit{et al (\cite[page 486]{annals})}, Bohnenblust and Hille \textquotedblleft\emph{showed, through a highly nontrivial argument, that the exponent $\frac{2m}{m+1}$ cannot be improved}\textquotedblright\ or according to Defant and Schwarting \cite[page 90]{DS}, Bohnenblust and Hille showed
\textquotedblleft\emph{with a sophisticated argument that the exponent $\frac{2m}{m+1}$ is optimal}\textquotedblright. Our argument shows that the optimality of the exponent $\frac{2m}{m+1}$ is a straightforward corollary of the Kahane--Salem--Zygmund inequality; in fact, as it shall be clear in the text, we prove a formally stronger result. In the second subsection we sketch
some ideas that may be useful in the investigation of lower bounds for the optimal constants of the complex polynomial Bohnenblust--Hille inequality; finally, in the third subsection we show how the Bohnenblust--Hille inequality can be used to prove the optimality of the power $\frac{m+1}{2}$ in the Kahane--Salem--Zygmund inequality (including the case of real scalars). Section 11 deals with some open problems and directions. In a final {\em Appendix}, we adapt some of the techniques used along this paper to a wide range of parameters. More precisely, we can estimate the constants satisfying Bohnenblust--Hille type inequalities when the exponent $\frac{2n}{n+1}$ is replaced by any $q\in\left[\frac{2n}{n+1},\infty\right)$.

\section{The search for the optimal multilinear Bohnenblust--Hille constants}

A series of very recent works (see \cite{annals, jfa, diniz2, Mu, munn, ap,
ap2, jmaa, Diana}) have investigated estimates for $C_{n}$ for the polynomial
and multilinear cases. The first estimates for the constants $C_{n}$ indicate
that one should expect an exponential growth for the optimal constants
$\left(  K_{n}\right)  _{n=1}^{\infty}$ satisfying the multilinear
Bohnenblust--Hille inequality:

\medskip

\begin{itemize}
\item $K_{n}\leq n^{\frac{n+1}{2n}}2^{\frac{n-1}{2}}$ \ (\cite{bh}, 1931),

\item $K_{n}\leq2^{\frac{n-1}{2}}$ \ (\cite{d, Ka}, 1970's),

\item $K_{n}\leq\left(  \frac{2}{\sqrt{\pi}}\right)  ^{n-1}$ \ (\cite{Q}, 1995).
\end{itemize}

\medskip

It is worth mentioning that the Bohnenblust--Hille inequality also holds for
the case of real scalars. In this paper, for the sake of simplicity, we shall
first work with real scalars. As a matter of fact, since the upper estimates
\eqref{orig} also hold for the complex case (because these estimates are
clearly bigger than the best known estimates for the complex case (see
\cite{jmaa})) our whole procedure encompasses both the real and complex cases.
For the sake of completeness (and since the complex case is the most important
for applications) in Section 9 we present separate estimates for the complex case.

Up to now the optimal values of these constants are unknown (for details see
\cite[Remark i, page 178]{Blei} or \cite{jfa, jmaa} and references therein).
Only very recently (see \cite{Mu}) quite surprising results were proved and
new connections with different subjects have arisen:

\begin{itemize}
\item[(i)] the sequence $\left(  K_{n}\right)  _{n=1}^{\infty}$ has a
subexponential growth (\cite{Mu}),

\item[(ii)] the sequence $\left(  K_{n}\right)  _{n=1}^{\infty}$ does not have
a polynomial expression (\cite{ap}),

\item[(iii)] if $q>0.526322$ (real case) or $q>0.304975$ (complex case), then
$K_{n}\nsim n^{q}$ (\cite{ap, ap2}).

\item[(iv)] the exact growth of $K_{n}$ is related to a conjecture of Aaronson
and Ambainis \cite{AA} about classical simulations of quantum query algorithms
(\cite{diniz2}).
\end{itemize}

Notwithstanding the recent advances a lot of mystery remains on the estimates
of the optimal constants satisfying the multilinear (and polynomial)
Bohnenblust--Hille inequality. Even simple questions remain without solution:

\begin{itemize}
\item \textbf{Problem 1.} Is $\left(  K_{n}\right)  _{n=1}^{\infty}$ increasing?

\item \textbf{Problem 2.} Does $\left(  K_{n}\right)  _{n=1}^{\infty}$ have a
\textquotedblleft well behaved\textquotedblright\ growth?
\end{itemize}

These two questions (mainly Problem 2), whose answers are quite likely
positive (but unfortunately unknown), are crucial barriers for the achievement
of stronger results on the behavior of the optimal constants. For example the
possibility of strong fluctuations on the optimal constants seems to be a
barrier to directly conclude (from (ii) above) that the optimal constants have
a subpolynomial growth. The problems above are well-characterized by the
Dichotomy Theorem (recently obtained in \cite{ap}).

\subsection{The Dichotomy Theorem}

In \cite{ap}, which can be considered a continuation of \cite{Mu}, a
\emph{dichotomy theorem} for the candidates of constants satisfying the
multilinear Bohnenblust--Hille inequality is proved and, as a consequence,
provides some new information on the optimal constants. In \cite{ap} a
sequence of positive real numbers $\left(  R_{n}\right)  _{n=1}^{\infty}$ is
said to be \textit{well-behaved} if there are $L_{1},L_{2}\in\lbrack0,\infty]
$ such that%
\[
\lim_{n\rightarrow\infty}\frac{R_{2n}}{R_{n}}=L_{1}\text{ }%
\]
and%
\[
\lim_{n\rightarrow\infty}\left(  R_{n+1}-R_{n}\right)  =L_{2}.
\]

The following result from \cite{ap}, in our opinion, is a good description of
the main obstacles that appear in the search of the optimal constants:

\begin{theorem}
[Dichotomy Theorem \cite{ap}]The sequence of optimal constants $\left(
K_{n}\right)  _{n=1}^{\infty}$ satisfying the Bohnenblust--Hille inequality
satisfies one and only one of the following assertions:

\begin{itemize}
\item[(i)] It is subexponential and not well-behaved.

\item[(ii)] It is well-behaved with
\[
\lim_{n\rightarrow\infty}\frac{K_{2n}}{K_{n}}\in\left[  1,\frac{e^{1-\frac
{1}{2}\gamma}}{\sqrt{2}}\right]
\]
and
\[
\lim_{n\rightarrow\infty}\left(  K_{n+1}-K_{n}\right)  =0,
\]
where $\gamma$ denotes the Euler--Mascheroni constant 
$$\displaystyle \gamma=\lim_{m\rightarrow\infty}\left(  \left(  -\log m\right)  +\sum_{k=1}^{m}k^{-1}\right).$$
\end{itemize}
\end{theorem}

Having in mind the above result, our belief (and the common sense, we think)
is that the situation (ii) holds but, as a matter of fact, a proof of this
fact seems to be far from the actual state-of-the-art of the subject. One of
the main contributions of the present paper shows that
\[
K_{n+1}-K_{n}<\frac{0.87}{n^{0.473}}%
\]
for infinitely many values of $n\in\mathbb{N}$. The central tool for proving
the above estimate and related theorems is a result of independent interest
which uncovers part of the uncertainties related to the subject: there exists
a sequence $\left(  R_{n}\right)  _{n=1}^{\infty}$ satisfying the multilinear
Bohnenblust--Hille inequality such that
\begin{equation}
\lim_{n\rightarrow\infty}\left(  R_{n+1}-R_{n}\right)  =0. \label{aaw}%
\end{equation}

Although we do not solve Problems 1 and 2, our results shall allow us to
conclude, among other results, that the optimal multilinear Bohnenblust--Hille
constants do have a subpolynomial growth and, moreover, a sub $p$-harmonic
growth for $p\approx0.47$ in the real case and $p\approx0.69 $ in the complex
case (see Theorem \ref{tt55} and Section 9); the main contributions of this
paper shall be presented in Section 5.

\section{A chronological overview of recent results}

In view of the large amount of recent papers and preprints related to the
subject, we shall dedicate some space to locate the contribution of the
present paper in the current state-of-the-art of the subject.

\begin{itemize}
\item In (\cite{defant2}, 2009), the bilinear version of the
Bohnenblust--Hille inequality (known as Littlewood's $4/3$ theorem) is
explored in a new direction and this paper rediscovers the importance of the
Bohnenblust--Hille inequality.

\item The paper (\cite{defant}, 2011) is a remarkable work of A. Defant, D.
Popa and U. Schwarting providing a new proof of the Bohnenblust--Hille
inequality which also led to interesting vector-valued generalizations.\

\item In (\cite{annals}, 2011) it is proved that the polynomial
Bohnenblust--Hille inequality is hypercontractive. Several striking
applications are presented.

\item In (\cite{jmaa}, 2012) new constants satisfying the multilinear
Bohnenblust--Hille inequality are presented, based on the arguments of the new
proof of the Bohnenblust--Hille theorem from \cite{defant}. An improvement of
this approach (for the case of complex scalars) is presented in (\cite{ap2}, 2012).

\item In (\cite{Mu}, 2012) some numerical investigations on the asymptotic
growth of the constants satisfying the multilinear Bohnenblust--Hille
inequality are presented; in this direction, in (\cite{jfa}, 2012) some
somewhat surprising results are obtained:

\medskip

\textbf{Theorem }(\cite{Mu}). There exists a sequence $\left(  Z_{n}\right)
_{n=1}^{\infty}$ satisfying the multilinear Bohnenblust--Hille inequality and
\[
\lim_{n\rightarrow\infty}\frac{Z_{n+1}}{Z_{n}}=1.
\]

\medskip

\textbf{Theorem }(\cite[Appendix]{Mu}). The optimal constants $\left(
K_{n}\right)  _{n=1}^{\infty}$ satisfying the multilinear Bohnenblust--Hille
inequality have a subexponential growth. In particular, if there is a constant
$L>0$ so that
\[
\lim_{n\rightarrow\infty}\frac{K_{n+1}}{K_{n}}=L,
\]
then
\[
L=1.
\]

\item In (\cite{ap}, 2012) a Dichotomy Theorem is proved and, as a
consequence, for example, it is shown that the optimal constants satisfying
the multilinear Bohnenblust--Hille inequality do not have a polynomial expression.

\item In (\cite{diniz2, munn}, 2012), in a completely different line of
attack, the authors obtain lower bounds for the constants of the multilinear
and polynomial Bohnenblust--Hille inequalities.

\item In (\cite{Diana}, 2012) an explicit formula for some recursive formulae
for constants satisfying the multilinear Bohnenblust--Hille inequality (from
\cite{jfa, jmaa}) is obtained (the original formulae on \cite{jfa, jmaa} were
obtained via a complicated recursive formula).
\end{itemize}

\section{The Fundamental Lemma\label{w3}}

We need to recall some notation. We shall work with the case of real scalars
but, as mentioned before, the same results hold in the case of complex scalars.

As earlier, the Greek letter $\gamma$ shall denote the Euler--Mascheroni
constant,
$$\gamma=\lim_{m\rightarrow\infty}\left(  -\log m+\sum_{k=1}^{m}\frac{1}{k}\right)  \approx 0.5772.$$
Also, henceforth, we use the notation
\begin{equation}
A_{p}:=\sqrt{2}\left(  \frac{\Gamma(\frac{p+1}{2})}{\sqrt{\pi}}\right)
^{1/p},\label{kkklll}%
\end{equation}
for $p>p_{0}\approx1.847$ and
\begin{equation}
A_{p}:=2^{\frac{1}{2}-\frac{1}{p}}\label{nob}%
\end{equation}
for $p\leq p_{0}\approx1.847.$ The precise definition of $p_{0}$ is the
following: $p_{0}\in(1,2)$ is the unique real number with
\[
\Gamma\left(  \frac{p_{0}+1}{2}\right)  =\frac{\sqrt{\pi}}{2}.
\]
The constants $A_{p}$ are precisely the best constants satisfying Khinchine's
inequality (these constants are due to U. Haagerup \cite{Ha}). In \cite{jmaa}
it was proved that the following constants satisfy the multilinear
Bohnenblust--Hille inequality:
\begin{equation}
C_{m}=\left\{
\begin{array}
[c]{ll}%
1 & \text{if }m=1,\\
\left(  A_{\frac{2m}{m+2}}^{m/2}\right)  ^{-1}C_{\frac{m}{2}} & \text{if
}m\text{ is even, and }\\
\left(  A_{\frac{2m-2}{m+1}}^{\frac{-1-m}{2}}C_{\frac{m-1}{2}}\right)
^{\frac{m-1}{2m}}\left(  A_{\frac{2m+2}{m+3}}^{\frac{1-m}{2}}C_{\frac{m+1}{2}%
}\right)  ^{\frac{m+1}{2m}} & \text{if }m\text{ is odd.}%
\end{array}
\right.  \label{orig}%
\end{equation}

From now on, $C_{m}$ shall always stand for the constants in \eqref{orig}. Up
to now these are the best (smallest) known constants satisfying the (real)
multilinear Bohnenblust--Hille inequality (these constants also work for the
complex case, although in this case even smaller constants are known; see
Section 9). It was not known if the sequence $\left(  C_{m}\right)
_{m=1}^{\infty}$ is increasing; in \cite{jfa} it was proved that if the above
sequence $\left(  C_{m}\right)  _{m=1}^{\infty}$ is increasing, then
\begin{equation}
\lim_{m\rightarrow\infty}\frac{C_{m}}{C_{m-1}}=1.\label{qw55}%
\end{equation}
If $\left(  C_{m}\right)  _{m=1}^{\infty}$ is not increasing, the sequence
\begin{equation}
C_{n}^{\prime}=\left\{
\begin{array}
[c]{ll}%
1 & \text{ if }n=1,\\
DC_{n/2}^{\prime} & \text{ if }n\text{ is even, and }\\
D\left(  C_{\frac{n-1}{2}}^{\prime}\right)  ^{\frac{n-1}{2n}}\left(
C_{\frac{n+1}{2}}^{\prime}\right)  ^{\frac{n+1}{2n}} & \text{ for }n\text{
odd,}%
\end{array}
\right.  \label{yt}%
\end{equation}
is such that
\[
\lim_{n\rightarrow\infty}\frac{C_{n+1}^{\prime}}{C_{n}^{\prime}}=1.
\]
Above, $D$ (whose precise value was not known) is any common upper bound for
\begin{equation}
\left(  A_{\frac{2m}{m+2}}^{-m/2}\right)  _{m=1}^{\infty}\text{ }\label{ty}%
\end{equation}
and
\begin{equation}
\left(  \left(  A_{\frac{2m-2}{m+1}}^{\frac{-1-m}{2}}\right)  ^{\frac{m-1}%
{2m}}\left(  A_{\frac{2m+2}{m+3}}^{\frac{1-m}{2}}\right)  ^{\frac{m+1}{2m}%
}\right)  _{m=1}^{\infty}.\label{ty77}%
\end{equation}
In \cite{jfa} it was also proved that both sequences tend to $\frac
{e^{1-\frac{1}{2}\gamma}}{\sqrt{2}}\approx1.4403$ but no information about
their eventual monotonicity is provided. To summarize, in \cite{Mu} is shown
that there exists a sequence of constants $\left(  Z_{m}\right)
_{m=1}^{\infty}$ satisfying the multilinear Bohnenblust--Hille inequality and
so that
\[
\lim_{n\rightarrow\infty}\frac{Z_{n+1}}{Z_{n}}=1,
\]
but the precise formula of the constants $Z_{m}$ depends on the (unknown) value of $D$ or, of course, on the (unknown) monotonicity of the constants \eqref{orig}.

In the present paper, as preparatory lemmata, we solve both problems by
proving that:

\begin{itemize}
\item[(i.-)] The sequence given in \eqref{orig} is increasing.

\item[(ii.-)] $D=$ $\frac{e^{1-\frac{1}{2}\gamma}}{\sqrt{2}}\approx1.4403$
(and, of course, this value is sharp).
\end{itemize}

This information has useful consequences. The fact that $D<2$ shall be crucial
for the proof of our first result (Theorem \ref{t4}), which we call
Fundamental Lemma.

The concrete estimate for $D$ allows us to deal with a simple presentation of
good (small) estimates for the constants of the multilinear Bohnenblust--Hille
inequality. More precisely (using the value of $D$) now we know that the
sequence
\begin{equation}
S_{n}=\left\{
\begin{array}[c]{ll}
\left(  \sqrt{2}\right)  ^{n-1} & \text{ if }n=1,2\\
\left(  \frac{e^{1-\frac{1}{2}\gamma}}{\sqrt{2}}\right)  S_{n/2} & \text{ for
}n\text{ even,}\\
\left(  \frac{e^{1-\frac{1}{2}\gamma}}{\sqrt{2}}\right)  \left(  S_{\frac
{n-1}{2}}\right)  ^{\frac{n-1}{2n}}\left(  S_{\frac{n+1}{2}}\right)
^{\frac{n+1}{2n}} & \text{ for }n\text{ odd}
\end{array}
\right.  \label{di11}
\end{equation}
satisfies the multilinear Bohnenblust--Hille inequality. This estimate for $D$ can also be used in the explicit formula for the constants \eqref{di11} presented in \cite{Diana}.

The sequence $\left(  R_{n}\right)  _{n=1}^{\infty}$ in the Fundamental Lemma
is a slight modification of the sequence \eqref{di11}. A natural question is
why not to work directly with the sequences \eqref{yt} or \eqref{di11}? The
main reason is that, having in mind the applications related to the optimal
constants provided in this paper, in fact we need to quantify how
$R_{n+1}-R_{n}$ tends to zero, and the direct estimation of how $C_{n+1}-C_{n}
$ or $S_{n+1}-S_{n}$ tend to zero is not a good approach. It is important to
notice that, as it shall be clear later, this slight modification keeps the
essence of the sequence $\left(  C_{n}\right)  _{n=1}^{\infty}$ in the sense
that it does not modify its asymptotic growth.

\section{Summary of the main results\label{ss11}}

The proof of the Fundamental Lemma furnishes concrete information on the
optimal constants satisfying the multilinear Bohnenblust--Hille inequality.
Our constructive approach provides an explicit sequence of constants with the
desired property. We also estimate how the difference $R_{n+1}-R_{n}$ tends
(monotonely) to $0^{+}$. In fact we have
\begin{equation}
R_{n+1}-R_{n}<\frac{0.87}{n^{0.473678}} \label{ee22}%
\end{equation}
for every positive integer $n.$ More precisely our constants are so that
\begin{equation}
R_{n+1}-R_{n}\leq\left(  2\sqrt{2}-4e^{\frac{1}{2}\gamma-1}\right)
n^{\log_{2}\left(  2^{-3/2}e^{1-\frac{1}{2}\gamma}\right)  }. \label{ee23}%
\end{equation}
The estimates \eqref{ee22}, \eqref{ee23} are crucial for the applications to
the optimal constants. Without our approach (working directly with the
sequences obtained in \cite{jfa, ap2, jmaa}) it would be rather difficult to
achieve the same results due their forbidding recursive formulae of the
previous sequences. Even the closed (explicit) formula for the multilinear
Bohnenblust--Hille constants presented in \cite{Diana} lodges some technical
difficulties when estimating the difference $S_{n+1}-S_{n}.$

We also stress that in all previous related papers there was not available
information on the monotonicity of the limits involving the Gamma function and
this lack of information was a peremptory barrier for estimating
$C_{n+1}-C_{n}.$

The constants $\left(  R_{n}\right)  _{n=1}^{\infty}$ that we obtain here with
the property \eqref{aaw} are slightly bigger than the constants from
\cite{jfa, ap2, jmaa} but, on the other hand, they are constructed in a more
simple fashion so that with a careful control of the monotonicity of the
expressions involving the Gamma Function, we are finally able to quantify how
far $R_{n+1}-R_{n}$ approaches to zero. As it shall be shown, although
$C_{n}\leq R_{n},$ these sequences have essentially the same asymptotic
behavior. The main results of this paper are the following consequences of the
above results:

\begin{itemize}
\item \emph{(Theorem \ref{56})} Let $\left(  K_{n}\right)  _{n=1}^{\infty}$ be
the sequence of best constants satisfying the multilinear Bohnenblust--Hille
inequality. If there is an $L\in\lbrack-\infty,\infty]$ so that
\[
\lim_{n\rightarrow\infty}\left(  K_{n+1}-K_{n}\right)  =L,
\]
then
\[
L=0.
\]

\item \emph{(Theorem \ref{uty} and Section \ref{complex})} Let $\left(
K_{n}\right)  _{n=1}^{\infty}$ be the optimal constants of the multilinear
Bohnenblust--Hille inequality. For any $\varepsilon>0$, we have
\begin{align*}
K_{n+1}-K_{n}  &  <\left(  2\sqrt{2}-4e^{\frac{1}{2}\gamma-1}\right)
n^{\log_{2}\left(  2^{-3/2}e^{1-\frac{1}{2}\gamma}\right)  +\varepsilon}\text{
(real scalars)}\\
K_{n+1}-K_{n}  &  <\left(  \frac{4}{\sqrt{\pi}}-\frac{4}{e^{\frac{1}{2}%
-\frac{1}{2}\gamma}\sqrt{\pi}}\right)  n^{\log_{2}\left(  \frac{e^{\frac{1}%
{2}-\frac{1}{2}\gamma}}{2}\right)  +\varepsilon}\text{ (complex scalars)}%
\end{align*}
for infinitely many $n$'s. Numerically, choosing a sufficiently small
$\varepsilon>0$,
\begin{align*}
K_{n+1}-K_{n}  &  <\frac{0.87}{n^{0.473678}}\text{ (real scalars)}\\
K_{n+1}-K_{n}  &  <\frac{0.44}{n^{0.695025}}\text{ (complex scalars)}%
\end{align*}

\item \emph{(Corollary \ref{t6})} \label{lema1222} The optimal multilinear
Bohnenblust--Hille constants $\left(  K_{n}\right)  _{n=1}^{\infty}$ satisfy
\[
\lim_{n}\inf\left(  K_{n+1}-K_{n}\right)  \leq0.
\]

\item \emph{(Theorem \ref{tt55} and Section \ref{complex})} The optimal
multilinear Bohnenblust--Hille constants $\left(  K_{n}\right)  _{n=1}%
^{\infty} $ satisfy
\begin{align*}
K_{n}  &  <1+0.87\cdot\sum_{j=1}^{n-1}\frac{1}{j^{0.473678}}\text{ (real
scalars)}\\
K_{n}  &  <1+0.44\cdot{\sum\limits_{j=1}^{n-1}}\frac{1}{j^{0.695025}}\text{
(complex scalars)}%
\end{align*}
for every $n\geq2$.

\item \emph{(Corollary \ref{nnhh} and Section \ref{complex})} The optimal
multilinear Bohnenblust--Hille constants $\left(  K_{n}\right)  _{n=1}%
^{\infty} $ satisfy%
\begin{align*}
K_{n}  &  <1.65\left(  n-1\right)  ^{0.526322}+0.13\text{ (real scalars)}\\
K_{n}  &  <1.41\left(  n-1\right)  ^{0.304975}-0.04\text{ (complex scalars)}%
\end{align*}
for every $n\geq2$.
\end{itemize}

The above results complement and complete recent information given in
\cite{ap}.

\section{First results: technical lemmata\label{p0}}

Our first result, and crucial for our goals, is the proof that the sequence
$\left(  A_{\frac{2m}{m+2}}^{-m/2}\right)  _{m=1}^{\infty}$ is increasing. We
stress that this is not an obvious result. In fact, since the sequence
$\left(  A_{p}\right)  _{p\geq1}$ is composed by the best constants satisfying
the Khinchine inequality, using the monotonicity of the $L_{p}$-norms we can
conclude that
\[
\left(  A_{\frac{2m}{m+2}}\right)  _{m=1}^{\infty}\subset(0,1)
\]
is increasing. Hence
\[
\left(  A_{\frac{2m}{m+2}}^{-1}\right)  _{m=1}^{\infty}\subset(1,\infty)
\]
is decreasing; thus, since $\left(  m/2\right)  _{m=1}^{\infty}$ is
increasing, no straightforward conclusions on the monotonicity of $\left(
A_{\frac{2m}{m+2}}^{-m/2}\right)  _{m=1}^{\infty}$ can be inferred. The key
result used in the proof of the following lemmata is an useful theorem due to
F. Qi \cite{qi} asserting that
\[
\left(  \frac{\Gamma\left(  s\right)  }{\Gamma\left(  r\right)  }\right)
^{\frac{1}{s-r}}%
\]
increases with $r,s>0.$

\begin{lemma}
\label{lema1}The sequence $\left(  A_{\frac{2m}{m+2}}^{-m/2}\right)
_{m=1}^{\infty}$ is increasing. In particular
\[
C_{2m}\leq\left(  \frac{e^{1-\frac{1}{2}\gamma}}{\sqrt{2}}\right)  C_{m}%
\]
for all $m.$
\end{lemma}

\begin{proof}
Since
\[
\frac{2m}{m+2}>p_{0}\approx1.847
\]
for all $m\geq25,$ the formula \eqref{kkklll} holds only for $m\geq25;$ but a
direct inspection (using \eqref{nob}) shows that the sequence is increasing
for $m<25.$

For $m\geq25$, note that
\[
A_{\frac{2m}{m+2}}^{-m/2}=\frac{1}{\sqrt{2}}\left(  \frac{\Gamma\left(
\frac{3m+2}{2m+4}\right)  }{\Gamma\left(  \frac{3}{2}\right)  }\right)
^{\frac{m+2}{-4}}.
\]
But, from \cite[Theorem 2]{qi} we know that
\[
\left(  \left(  \frac{\Gamma\left(  \frac{3m+2}{2m+4}\right)  }{\Gamma\left(
\frac{3}{2}\right)  }\right)  ^{\frac{m+2}{-2}}\right)  _{m=1}^{\infty}%
\]
is increasing and the conclusion is immediate.
\end{proof}

A first consequence of this lemma solves a question left open in \cite{jfa}.

\begin{proposition}
The sequence
\[
C_{n}=\left\{
\begin{array}
[c]{ll}%
1 & \text{ if } n=1,\\
\left(  A_{\frac{2n}{n+2}}^{n/2}\right)  ^{-1}C_{\frac{n}{2}} & \text{ if } n
\text{ is even, and }\\
\left(  A_{\frac{2n-2}{n+1}}^{\frac{-1-n}{2}}C_{\frac{n-1}{2}}\right)
^{\frac{n-1}{2n}}\left(  A_{\frac{2n+2}{n+3}}^{\frac{1-n}{2}}C_{\frac{n+1}{2}%
}\right)  ^{\frac{n+1}{2n}} & \text{ if } n\text{ is odd}%
\end{array}
\right.
\]
is increasing.
\end{proposition}

\begin{proof}
We proceed by induction (the first values can be directly checked). Let us
suppose that the result is valid for all positive integers smaller than $n-1$
and, then, use induction.

\textbf{First case. }$n$ is even.

Note that
\[
C_{n}\leq C_{n+1}%
\]
if and only if
\[
\frac{C_{n/2}}{A_{\frac{2n}{n+2}}^{{\small n/2}}}\leq\left(  \frac{C_{n/2}%
}{A_{\frac{2n}{n+2}}^{\left(  {\small n+2}\right)  {\small /2}}}\right)
^{\frac{n}{2\left(  n+1\right)  }}.\left(  \frac{C_{\frac{n+2}{2}}}%
{A_{\frac{2n+4}{n+4}}^{{\small n/2}}}\right)  ^{\frac{n+2}{2\left(
n+1\right)  }}%
\]
and this is equivalent to
\[
\left(  C_{n/2}\right)  ^{\frac{n+2}{2\left(  n+1\right)  }}\left(  \left(
A_{\frac{2n}{n+2}}^{{\small n/2}}\right)  ^{-1}\right)  ^{\frac{n}{2\left(
n+1\right)  }}\leq\left(  C_{\frac{n+2}{2}}\right)  ^{\frac{n+2}{2\left(
n+1\right)  }}\left(  \left(  A_{\frac{2n+4}{n+4}}^{\left(  {\small n+2}%
\right)  {\small /2}}\right)  ^{-1}\right)  ^{\frac{n}{2\left(  n+1\right)  }%
}.
\]
But the last inequality is true. In fact, from the induction hypothesis we
have
\[
C_{n/2}\leq C_{\frac{n+2}{2}}%
\]
and from Lemma \ref{lema1} we know that
\[
\left(  A_{\frac{2n}{n+2}}^{{\small n/2}}\right)  ^{-1}\leq\left(
A_{\frac{2n+4}{n+4}}^{\left(  {\small n+2}\right)  {\small /2}}\right)  ^{-1}%
\]
holds.

\textbf{Second case. }$n$ is odd.

A similar argument shows that
\[
C_{n}\leq C_{n+1}%
\]
if and only if%
\[
\frac{\left(  C_{\left(  n-1\right)  /2}\right)  ^{\frac{n-1}{2n}}}{\left(
A_{\frac{2n-2}{n+1}}^{\left(  {\small n-1}\right)  {\small /2}}\right)
^{\frac{n+1}{2n}}}\leq\frac{\left(  C_{\left(  n+1\right)  /2}\right)
^{\frac{n-1}{2n}}}{\left(  A_{\frac{2n+2}{n+3}}^{\left(  {\small n+1}\right)
{\small /2}}\right)  ^{\frac{n+1}{2n}}}%
\]
and this inequality is true using the induction hypothesis and Lemma
\ref{lema1}.
\end{proof}

\begin{lemma}
The sequence
\[
\left(  \left(  \left(  A_{\frac{2m+2}{m+3}}^{\frac{m-1}{2}}\right)
^{-1}\right)  ^{\frac{m+1}{2m}}.\left(  \left(  A_{\frac{2m-2}{m+1}}%
^{\frac{m+1}{2}}\right)  ^{-1}\right)  ^{\frac{m-1}{2m}}\right)
_{m=1}^{\infty}%
\]
is bounded by%
\[
D:=\left(  \frac{e^{1-\frac{1}{2}\gamma}}{\sqrt{2}}\right)  .
\]

\end{lemma}

\begin{proof}
Let
\[
X_{m}:=A_{\frac{2m}{m+2}}^{-m/2}%
\]
for all $m.$ From Lemma \ref{lema1} we know that $\left(  X_{m}\right)
_{m=1}^{\infty}$ is increasing and bounded by $D$. Note that
\[
\left(  \left(  A_{\frac{2m-2}{m+1}}^{\frac{m-1}{2}}\right)  ^{-1}\right)
=X_{m-1}\leq X_{m+1}=\left(  \left(  A_{\frac{2m+2}{m+3}}^{\frac{m+1}{2}%
}\right)  ^{-1}\right)  .
\]
Thus we have
\begin{align*}
&  \left(  \left(  A_{\frac{2m+2}{m+3}}^{\frac{m-1}{2}}\right)  ^{-1}\right)
^{\frac{m+1}{2m}}.\left(  \left(  A_{\frac{2m-2}{m+1}}^{\frac{m+1}{2}}\right)
^{-1}\right)  ^{\frac{m-1}{2m}}\\
&  =\left(  \left(  A_{\frac{2m+2}{m+3}}^{\frac{m+1}{2}}\right)  ^{-1}\right)
^{\frac{m-1}{2m}}.\left(  \left(  A_{\frac{2m-2}{m+1}}^{\frac{m-1}{2}}\right)
^{-1}\right)  ^{\frac{m+1}{2m}}\\
&  =\left(  X_{m+1}\right)  ^{\frac{m-1}{2m}}\left(  X_{m-1}\right)
^{\frac{m+1}{2m}}\\
&  \leq X_{m+1}.
\end{align*}
Since $\left(  X_{m}\right)  _{m=1}^{\infty}$ is increasing and bounded by $D$
we conclude that
\[
\left(  \left(  \left(  A_{\frac{2m+2}{m+3}}^{\frac{m-1}{2}}\right)
^{-1}\right)  ^{\frac{m+1}{2m}}.\left(  \left(  A_{\frac{2m-2}{m+1}}%
^{\frac{m+1}{2}}\right)  ^{-1}\right)  ^{\frac{m-1}{2m}}\right)
_{m=1}^{\infty}%
\]
is also bounded by $D.$
\end{proof}

\section{The proof of the Fundamental Lemma\label{w4}}

In this section we prove the Fundamental Lemma. We note that $\left(
S_{n}\right)  _{n=1}^{\infty}$ (defined in \eqref{di11}) is increasing and
satisfies the multilinear Bohnenblust--Hille inequality. The proof of the
first assertion is straightforward; for the proof of the second assertion we
just need to observe that $C_{n}\leq S_{n}$ for all $n.$ We recall that a
closed formula for the constants $\left(  S_{n}\right)  _{n=1}^{\infty}$ with
a generic $D$ in the place of $\left(  \frac{e^{1-\frac{1}{2}\gamma}}{\sqrt
{2}}\right)  $ appears in \cite{Diana}. Using the previous lemmata, the new
sequence defined by%
\[
M_{n}=\left\{
\begin{array}
[c]{ll}%
\left(  \sqrt{2}\right)  ^{n-1} & \text{ if }n=1,2\\
\left(  \frac{e^{1-\frac{1}{2}\gamma}}{\sqrt{2}}\right)  M_{\frac{n}{2}} &
\text{ if }n\text{ is even, and}\\
\left(  \frac{e^{1-\frac{1}{2}\gamma}}{\sqrt{2}}\right)  M_{\frac{n+1}{2}} &
\text{ if }n\text{ is odd}%
\end{array}
\right.
\]
is so that
\[
C_{n}\leq S_{n}\leq M_{n}%
\]
and a \textquotedblleft uniform perturbation\textquotedblright\ of this
sequence $\left(  M_{n}\right)  _{n=1}^{\infty}$ shall be the desired
sequence. Let
\[
D:=\left(  \frac{e^{1-\frac{1}{2}\gamma}}{\sqrt{2}}\right)  \approx1.4403.
\]
and, for all $k\geq1$, consider
\[
B_{k}:=\{2^{k-1}+1,\ldots,2^{k}\}.
\]

\begin{remark}
It is easy to note that for all $n\geq2$ we have
\[
M_{n}=\sqrt{2}D^{k-1}\text{ whenever }n\in B_{k}%
\]
and for this reason
\begin{equation}
\lim_{n\rightarrow\infty}\left(  M_{n+1}-M_{n}\right)  \label{ttrr}%
\end{equation}
does not exist.
\end{remark}

Since the limit \eqref{ttrr} does not exist, now consider the sequence
$\left(  R_{n}\right)  _{n=1}^{\infty}$, which is a slight uniform
perturbation of the sequence $\left(  M_{n}\right)  _{n=1}^{\infty}:$%
\begin{equation}
R_{n}:=\sqrt{2}\left(  D^{k-1}+\left(  j_{n}-1\right)  \left(  \frac
{D^{k}-D^{k-1}}{2^{k-1}}\right)  \right)  ,\text{ whenever }n\in B_{k}
\label{wwwd}%
\end{equation}
where $j_{n}$ is the position of $n$ in the order of the elements of $B_{k}.$

It is plain that
\[
M_{n}\leq R_{n}
\]
for all $n\geq3$ and, as we shall see,
\[
\left(  R_{n+1}-R_{n}\right)  _{n=1}^{\infty}
\]
is decreasing (monotone and non-increasing). Using the definition of $\left(
R_{n}\right)  _{n=1}^{\infty}$ with a careful handling of the expressions
involved it is not difficult to estimate how $R_{n+1}-R_{n}$ decreases to zero:

\begin{theorem}[The Fundamental Lemma]\label{t4}
The sequence \eqref{wwwd} satisfies the multilinear Bohnenblust--Hille inequality and $\left(  R_{n+1}-R_{n}\right)_{n=1}^{\infty}$ is decreasing and converges to zero. Moreover
\begin{equation}
R_{n+1}-R_{n}\leq\left(  2\sqrt{2}-4e^{\frac{1}{2}\gamma-1}\right) n^{\log_{2}\left(  2^{-3/2}e^{1-\frac{1}{2}\gamma}\right)  } \label{rhs}
\end{equation}
for all $n.$ Numerically,
\[
R_{n+1}-R_{n}<\frac{0.87}{n^{0.473678}}.
\]

\end{theorem}

\begin{proof}
Of course $\left(  R_{n}\right)  _{n=1}^{\infty}$ satisfies the multilinear Bohnenblust--Hille inequality. Let us show that $\left(  R_{n+1}-R_{n}\right)_{n=1}^{\infty}$ is decreasing. In fact, if $n\in B_{k}$, we have two possibilities:

\noindent\textbf{First case:} $n+1\in B_{k}.$ In this case
\begin{align*}
&  R_{n+1}-R_{n}=\\
&  =\sqrt{2}D^{k-1}+\sqrt{2}\left(  j_{n+1}-1\right)  \left(  \frac{D^{k}-D^{k-1}}{2^{k-1}}\right)  -\left(  \sqrt{2}D^{k-1}+\sqrt{2}\left(
j_{n}-1\right)  \left(  \frac{D^{k}-D^{k-1}}{2^{k-1}}\right)  \right) \\
&  =\sqrt{2}\left(  \frac{D^{k}-D^{k-1}}{2^{k-1}}\right).
\end{align*}

\noindent\textbf{Second case:} $n+1\in B_{k+1}.$ Here, $n=2^{k}$ and
$n+1=2^{k}+1$ and, thus,
\begin{align*}
&  R_{n+1}-R_{n}=\\
&  =\sqrt{2}D^{k}+\sqrt{2}\left(  1-1\right)  \left(  \frac{D^{k+1}-D^{k}%
}{2^{k}}\right)  \text{ }-\left(  \sqrt{2}D^{k-1}+\sqrt{2}\left(
2^{k-1}-1\right)  \left(  \frac{D^{k}-D^{k-1}}{2^{k-1}}\right)  \right) \\
&  =\sqrt{2}\left(  \frac{D^{k}-D^{k-1}}{2^{k-1}}\right)  .
\end{align*}
But, since $D<2$, we have
\[
\frac{D^{k}-D^{k-1}}{2^{k-1}}>\frac{D^{k+1}-D^{k}}{2^{k}}%
\]
and we conclude that $\left(  R_{n+1}-R_{n}\right)  _{n=1}^{\infty}$ is
decreasing. Now, if we consider the subsequence
\[
\left(  R_{2^{k}+1}-R_{2^{k}}\right)  _{k=1}^{\infty},
\]
we obtain
\begin{align}
\lim_{k\rightarrow\infty}\left(  R_{2^{k}+1}-R_{2^{k}}\right)   &  =\sqrt
{2}\lim_{k\rightarrow\infty}\left(  \frac{D^{k}-D^{k-1}}{2^{k-1}}\right)
\label{s4}\\
&  =\sqrt{2}\left(  D-1\right)  \lim_{k\rightarrow\infty}\left(  \frac{D}%
{2}\right)  ^{k-1}\nonumber\\
&  =0.\nonumber
\end{align}
Hence
\[
\lim_{n\rightarrow\infty}R_{n+1}-R_{n}=0.
\]

Next, let us estimate the difference $R_{n+1}-R_{n}.$ Let $k$ be such that
$n\in B_{k};$ we thus have
\[
2^{k-1}+1\leq n\leq2^{k}%
\]
and
\[
\log_{2}\left(  \frac{n}{2}\right)  \leq\log_{2}\left(  2^{k-1}\right)  =k-1.
\]
Using again that $D<2$ we conclude that
\[
R_{n+1}-R_{n}\leq\left(  \frac{D}{2}\right)  ^{k-1}\sqrt{2}\left(  D-1\right)
\leq\left(  \frac{D}{2}\right)  ^{\log_{2}\left(  \frac{n}{2}\right)  }%
\sqrt{2}\left(  D-1\right)
\]

and a direct calculation gives us
\[
R_{n+1}-R_{n}\leq\left(  2\sqrt{2}-4e^{\frac{1}{2}\gamma-1}\right)
n^{\log_{2}\left(  2^{-3/2}e^{1-\frac{1}{2}\gamma}\right)  }.
\]

\end{proof}

As we know, the constants defined in \eqref{wwwd} are slightly bigger than the
constants from \eqref{yt}, \eqref{di11}; but we stress that there seems to be
no damage, asymptotically speaking. More precisely, the limits of $\left(
\frac{R_{2n}}{R_{n}}\right)  _{n=1}^{\infty}$ and $\left(  \frac{R_{n+1}%
}{R_{n}}\right)  _{n=1}^{\infty}$ are exactly the same of $\left(
\frac{C_{2n}}{C_{n}}\right)  _{n=1}^{\infty}$ and $\left(  \frac{C_{n+1}%
}{C_{n}}\right)  _{n=1}^{\infty}$ (see also the paragraph immediately above
the Corollary \ref{nnhh}):

\begin{proposition}
The sequence $\left(  \frac{R_{2n}}{R_{n}}\right)  _{n=1}^{\infty}$ is
decreasing and
\begin{equation}
\lim_{n\rightarrow\infty}\frac{R_{2n}}{R_{n}}=\left(  \frac{e^{1-\frac{1}%
{2}\gamma}}{\sqrt{2}}\right)  . \label{gn}%
\end{equation}
Also
\begin{equation}
\lim_{n\rightarrow\infty}\frac{R_{n+1}}{R_{n}}=1. \label{uuu}%
\end{equation}

\end{proposition}

\begin{proof}
The proof that $\left(  \frac{R_{2n}}{R_{n}}\right)  _{n=1}^{\infty}$ is
decreasing needs some care with the details, but is essentially
straightforward and we omit.

Let $k$ be so that $2n\in B_{k};$ then $j_{2n}$ is even. Also, we have $n\in
B_{k-1}$ and note that $j_{n}=\frac{j_{2n}}{2}$.$\ $\ Hence
\[
\frac{R_{2n}}{R_{n}}=\frac{\sqrt{2}\left(  D^{k-1}+\left(  j_{2n}-1\right)
\left(  \frac{D^{k}-D^{k-1}}{2^{k-1}}\right)  \right)  }{\sqrt{2}\left(
D^{k-2}+\left(  \frac{j_{2n}}{2}-1\right)  \left(  \frac{D^{k-1}-D^{k-2}%
}{2^{k-2}}\right)  \right)  }.
\]
Considering the subsequence given for $j_{2n}=2$ we have
\begin{align*}
\frac{D^{k-1}+\left(  2-1\right)  \left(  \frac{D^{k}-D^{k-1}}{2^{k-1}%
}\right)  }{D^{k-2}+\left(  1-1\right)  \left(  \frac{D^{k-1}-D^{k-2}}%
{2^{k-2}}\right)  }  &  =\frac{D^{k-1}+\left(  \frac{D^{k}-D^{k-1}}{2^{k-1}%
}\right)  }{D^{k-2}}\\
&  =\frac{2^{k-1}D^{k-1}+D^{k}-D^{k-1}}{2^{k-1}D^{k-2}}\\
&  =\frac{D^{k-2}\left(  2^{k-1}D+D^{2}-D\right)  }{2^{k-1}D^{k-2}}\\
&  =\frac{2^{k-1}D+D^{2}-D}{2^{k-1}}\overset{k\rightarrow\infty}%
{\longrightarrow}D.
\end{align*}
Combining this fact with the monotonicity of the sequence we obtain \eqref{gn}. The proof of \eqref{uuu} is obvious.
\end{proof}

\section{Main results: optimal constants\label{occ}}

In this section $\left(  R_{n}\right)  _{n=1}^{\infty}$ denotes the sequence defined in \eqref{wwwd}. As a consequence of Theorem \ref{t4} we have some new information on the growth of the optimal constants satisfying the multilinear Bohnenblust--Hille inequality. The first result complements (although not formally generalizes) recent results from \cite{ap}:

\begin{theorem}
\label{56}Let $\left(  K_{n}\right)  _{n=1}^{\infty}$ be the sequence of the
optimal constants satisfying the multilinear Bohnenblust--Hille inequality. If
there is a constant $M\in\lbrack-\infty,\infty]$ so that%
\[
\lim_{n\rightarrow\infty}\left(  K_{n+1}-K_{n}\right)  =M
\]
then $M=0$.
\end{theorem}

\begin{proof}
The case $M\in\lbrack-\infty,0)$ is clearly not possible. Let us first suppose
that $M\in\left(  0,\infty\right)  $. Let $n_{0}$ be a positive integer so
that
\[
n\geq n_{0}\Rightarrow K_{n+1}-K_{n}>\frac{M}{2}%
\]
and $n_{1}$ be a positive integer so that
\[
n\geq n_{1}\Rightarrow R_{n+1}-R_{n}<\frac{M}{4}.
\]
So, if $n\geq n_{2}:=\max\{n_{1},n_{0}\},$ then%
\[
K_{n}-K_{n_{2}}>\left(  \frac{M}{2}\right)  \left(  n-n_{2}\right)
\]
and%
\[
R_{n}-R_{n_{2}}<\left(  \frac{M}{4}\right)  \left(  n-n_{2}\right)  .
\]
Let $N>n_{2}$ be so that%
\[
\left(  \frac{M}{2}\right)  \left(  N-n_{2}\right)  +K_{n_{2}}>R_{n_{2}%
}+\left(  \frac{M}{4}\right)  \left(  N-n_{2}\right)  .
\]
Note that this is possible since%
\[
\left(  \frac{M}{2}\right)  \left(  n-n_{2}\right)  -\left(  \frac{M}%
{4}\right)  \left(  n-n_{2}\right)  \rightarrow\infty.
\]
For this $N$ we have%
\[
K_{N}>\left(  \frac{M}{2}\right)  \left(  N-n_{2}\right)  +K_{n_{2}}>R_{n_{2}%
}+\left(  \frac{M}{4}\right)  \left(  N-n_{2}\right)  >R_{N},
\]
which is a contradiction. The case $M=\infty$ is a simple adaptation of the
previous case.
\end{proof}

Now we prove a result that can be considered the main theorem of this paper:

\begin{theorem}
\label{uty}Let $\left(  K_{n}\right)  _{n=1}^{\infty}$ be the optimal
constants satisfying the multilinear Bohnenblust--Hille constants. For any
$\varepsilon>0$, we have
\begin{equation}
K_{n+1}-K_{n}<\left(  2\sqrt{2}-4e^{\frac{1}{2}\gamma-1}\right)  n^{\log
_{2}\left(  2^{-3/2}e^{1-\frac{1}{2}\gamma}\right)  +\varepsilon} \label{qz}%
\end{equation}
for infinitely many $n$'s.
\end{theorem}

\begin{proof}
From the previous results we know that
\[
R_{n+1}-R_{n}\leq\left(  2\sqrt{2}-4e^{\frac{1}{2}\gamma-1}\right)
n^{\log_{2}\left(  2^{-3/2}e^{1-\frac{1}{2}\gamma}\right)  }%
\]
for all $n$. Summing the above inequalities it is plain that
\begin{equation}
R_{n}\leq1+\left(  2\sqrt{2}-4e^{\frac{1}{2}\gamma-1}\right)  {\textstyle\sum
\limits_{j=1}^{n-1}}j^{\log_{2}\left(  2^{-3/2}e^{1-\frac{1}{2}\gamma}\right)
}. \label{iiu}%
\end{equation}
If $\varepsilon>0$, let us define
\[
T_{n}=1+\left(  2\sqrt{2}-4e^{\frac{1}{2}\gamma-1}\right)  {\textstyle\sum
\limits_{j=1}^{n-1}}j^{\log_{2}\left(  2^{-3/2}e^{1-\frac{1}{2}\gamma}\right)
+\varepsilon}.
\]
Then
\[
T_{n+1}-T_{n}=\left(  2\sqrt{2}-4e^{\frac{1}{2}\gamma-1}\right)  n^{\log
_{2}\left(  2^{-3/2}e^{1-\frac{1}{2}\gamma}\right)  +\varepsilon}%
\]
It is simple to show that the set
\[
A_{\varepsilon}:=\left\{  n:K_{n+1}-K_{n}<T_{n+1}-T_{n}\right\}
\]
is infinite. In fact, if $A_{\varepsilon}$ was finite, let $n_{\varepsilon}$
be its minimum. So, for all $n>n_{\varepsilon}$ we would have
\[
K_{n+1}-K_{n}\geq T_{n+1}-T_{n}.
\]
Also, for any $N>n_{\varepsilon}+1$, summing both sides for $n=n_{\varepsilon
}+1$ to $n=N,$ we have
\[
K_{N+1}-K_{n_{\varepsilon}+1}\geq T_{N+1}-T_{n_{\varepsilon}+1}.
\]
We finally obtain
\[
K_{N+1}-T_{N+1}\geq K_{n_{\varepsilon}+1}-T_{n_{\varepsilon}+1}%
\]
and it is a contradiction, since
\begin{align*}
&  K_{N+1}-T_{N+1}<R_{N+1}-T_{N+1}\leq\\
&  \leq\left(  2\sqrt{2}-4e^{\frac{1}{2}\gamma-1}\right)  {\textstyle\sum
\limits_{j=1}^{N}}j^{\log_{2}\left(  2^{-3/2}e^{1-\frac{1}{2}\gamma}\right)
}-\left(  2\sqrt{2}-4e^{\frac{1}{2}\gamma-1}\right)  {\textstyle\sum
\limits_{j=1}^{N}}j^{\log_{2}\left(  2^{-3/2}e^{1-\frac{1}{2}\gamma}\right)
+\varepsilon}%
\end{align*}
and this last expression tends to $-\infty$.
\end{proof}

Estimating the values in \eqref{qz} and choosing a sufficiently small $\varepsilon>0$ we can assert that
\[
K_{n+1}-K_{n}<\frac{0.87}{n^{0.473678}}%
\]
for infinitely many integers $n$. It seems quite likely that the optimal
constants of the multilinear Bohnenblust--Hille inequality have an uniform
growth. The above theorem induces us to conjecture that the estimate holds for
all $n$.

\begin{corollary}
\label{t6}The optimal multilinear Bohnenblust--Hille constants $\left(
K_{n}\right)  _{n=1}^{\infty}$ satisfy
\[
\lim_{n}\inf\left(  K_{n+1}-K_{n}\right)  \leq0.
\]

\end{corollary}

The following straightforward consequence of \eqref{iiu} seems to be of independent interest:

\begin{theorem}
\label{tt55}The optimal multilinear Bohnenblust--Hille constants $\left(
K_{n}\right)  _{n=1}^{\infty}$ satisfy
\[
K_{n}<1+0.87\cdot\sum_{j=1}^{n-1}\frac{1}{j^{0.473678}}%
\]
for every $n\geq2$.
\end{theorem}

We recall that in \cite{ap}, one of the consequences of the main theorem is
that
\begin{equation}
K_{n}\nsim n^{r}\text{ for all }r>q:=\log_{2}\left(  \frac{e^{1-\frac{\gamma
}{2}}}{\sqrt{2}}\right)  . \label{998}%
\end{equation}

The fact that our \textquotedblleft perturbation argument\textquotedblright%
\ does not cause any asymptotic damage is strongly corroborated by the
following generalization of \eqref{998}; note that the power of $n-1$ in \eqref{999} is exactly the number $q$ in \eqref{998}, although the approaches
are completely different:

\begin{corollary}
\label{nnhh}Let%
\[
C_{0}:=1+\left(  2^{\frac{3}{2}}-4e^{\frac{\gamma}{2}-1}\right)  \left(
\frac{2^{-1/2}e^{1-\frac{\gamma}{2}}}{2^{-1}-\log_{2}e^{1-\frac{\gamma}{2}}%
}+\left(  1+2^{\frac{-3}{2}}e^{1-\frac{\gamma}{2}}\right)  \right)
\approx0.122.
\]
The optimal multilinear Bohnenblust--Hille constants satisfy%
\begin{equation}
K_{n}<\left(  \frac{2^{\frac{5}{2}}-8e^{-1+\frac{\gamma}{2}}}{2\log_{2}\left(
e^{1-\frac{\gamma}{2}}\right)  -1}\right)  \left(  n-1\right)  ^{\log
_{2}\left(  \frac{e^{1-\frac{\gamma}{2}}}{\sqrt{2}}\right)  }+C_{0}
\label{999}%
\end{equation}
for all $n\geq2$. Numerically,%
\begin{equation}
K_{n}<1.65\left(  n-1\right)  ^{0.526322}+0.13. \label{y68}%
\end{equation}

\end{corollary}

\begin{proof}
Recall that%
\[
K_{n}<1+\left(  2\sqrt{2}-4e^{\frac{1}{2}\gamma-1}\right)  {\sum
\limits_{j=1}^{n-1}}j^{\log_{2}\left(  2^{-3/2}e^{1-\frac{1}{2}\gamma}\right)
}%
\]
For the sake of simplicity, let us write%
\[
p=-\log_{2}\left(  2^{-3/2}e^{1-\frac{1}{2}\gamma}\right)
\]
Note that, for $n\geq3,$ we can estimate ${\textstyle\sum\limits_{j=1}^{n-1}%
}\frac{1}{j^{p}}$ by
\begin{align*}
{\displaystyle\sum\limits_{j=1}^{n-1}}\frac{1}{j^{p}}  &  \leq
{\displaystyle\int\limits_{2}^{n-1}}x^{-p}dx+\left(  1+2^{-p}\right) \\
&  =\frac{1}{1-p}\left(  n-1\right)  ^{1-p}+\left(  \frac{2^{1-p}}%
{p-1}+\left(  1+2^{-p}\right)  \right)  .
\end{align*}
We thus have
\[
K_{n}<1+\left(  2\sqrt{2}-4e^{\frac{1}{2}\gamma-1}\right)  \left(  \frac
{1}{1-p}\left(  n-1\right)  ^{1-p}+\left(  \frac{2^{1-p}}{p-1}+\left(
1+2^{-p}\right)  \right)  \right)
\]
and a simple calculation gives us \eqref{999} and \eqref{y68}.
\end{proof}

\section{The complex case: When $\pi, e$ and $\gamma$ meet \label{complex}}

For complex scalars the best known constants satisfying the multilinear
Bohnenblust--Hille inequality are presented in \cite{ap2} by the formula%
\[
\widetilde{C}_{n}=\left\{
\begin{array}
[c]{ll}%
1 & \text{ if }n=1,\\
\left(  \left(  \widetilde{A_{\frac{2n}{n+2}}}\right)  ^{n/2}\right)
^{-1}\widetilde{C}_{\frac{n}{2}} & \text{ if }n\text{ is even, and }\\
\left(  \left(  \widetilde{A_{\frac{2n-2}{n+1}}}\right)  ^{\frac{-1-n}{2}%
}\widetilde{C}_{\frac{n-1}{2}}\right)  ^{\frac{n-1}{2n}}\left(  \left(
\widetilde{A_{\frac{2n+2}{n+3}}}\right)  ^{\frac{1-n}{2}}\widetilde{C}%
_{\frac{n+1}{2}}\right)  ^{\frac{n+1}{2n}} & \text{ if }n\text{ is odd,}%
\end{array}
\right.
\]
where%
\[
\widetilde{A_{p}}=\left(  \Gamma\left(  \frac{p+2}{2}\right)  \right)
^{\frac{1}{p}}.
\]

A similar procedure (using \cite[Theorem 2]{qi}) of that from Section \ref{w3}
proves that the sequence $\left(  \left(  \widetilde{A_{\frac{2m}{m+2}}%
}\right)  ^{-m/2}\right)  _{m=1}^{\infty}$ is increasing. We just need to use
$s=\frac{2m}{m+2}$ and $r=2$. In particular we conclude that
\[
\widetilde{C}_{2m}\leq\left(  e^{\frac{1}{2}-\frac{1}{2}\gamma}\right)
\widetilde{C}_{m}%
\]
for all $m.$ Also, still imitating the arguments from Section \ref{w3} we
prove that the sequence
\[
\left(  \left(  \left(  \widetilde{A_{\frac{2n-2}{n+1}}}\right)  ^{\frac
{-1-n}{2}}\widetilde{C}_{\frac{n-1}{2}}\right)  ^{\frac{n-1}{2n}}\left(
\left(  \widetilde{A_{\frac{2n+2}{n+3}}}\right)  ^{\frac{1-n}{2}}\widetilde
{C}_{\frac{n+1}{2}}\right)  ^{\frac{n+1}{2n}}\right)  _{n=1}^{\infty}%
\]
is bounded by
\[
\widetilde{D}:=\left(  e^{\frac{1}{2}-\frac{1}{2}\gamma}\right)  .
\]

In a similar fashion of what we did in the previous sections we thus conclude
that the sequence
\[
\widetilde{S}_{n}=\left\{
\begin{array}
[c]{ll}%
\left(  \frac{2}{\sqrt{\pi}}\right)  ^{n-1} & \text{ if }n=1,2\\
\left(  e^{\frac{1}{2}-\frac{1}{2}\gamma}\right)  \widetilde{S}_{n/2} & \text{
for }n\text{ even,}\\
\left(  e^{\frac{1}{2}-\frac{1}{2}\gamma}\right)  \left(  \widetilde{S}%
_{\frac{n-1}{2}}\right)  ^{\frac{n-1}{2n}}\left(  \widetilde{S}_{\frac{n+1}%
{2}}\right)  ^{\frac{n+1}{2n}} & \text{ for }n\text{ odd}%
\end{array}
\right.
\]
is increasing and satisfies the Bohnenblust-Hille inequality. Now we define
\[
\widetilde{M}_{n}=\left\{
\begin{array}
[c]{ll}%
\left(  \frac{2}{\sqrt{\pi}}\right)  ^{n-1} & \text{ if }n=1,2\\
\left(  e^{\frac{1}{2}-\frac{1}{2}\gamma}\right)  \widetilde{M}_{\frac{n}{2}}
& \text{ if }n\text{ is even, and}\\
\left(  e^{\frac{1}{2}-\frac{1}{2}\gamma}\right)  \widetilde{M}_{\frac{n+1}%
{2}} & \text{ if }n\text{ is odd}%
\end{array}
\right.
\]
and it is plain that
\[
\widetilde{C}_{n}\leq\widetilde{S}_{n}\leq\widetilde{M}_{n}.
\]
Considering again
\[
B_{k}:=\{2^{k-1}+1,\ldots,2^{k}\}
\]
for all $k\geq1$, we have
\[
\widetilde{M}_{n}=\frac{2}{\sqrt{\pi}}\widetilde{D}^{k-1}\text{, if }n\in
B_{k}%
\]
and define the uniform perturbation of $\widetilde{M}_{n}$:
\begin{equation}
\widetilde{R}_{n}:=\frac{2}{\sqrt{\pi}}\left(  \widetilde{D}^{k-1}+\left(
j_{n}-1\right)  \left(  \frac{\widetilde{D}^{k}-\widetilde{D}^{k-1}}{2^{k-1}%
}\right)  \right)  ,\text{ } \label{wwwd2}%
\end{equation}
where $n\in B_{k}$ and $j_{n}$ is the position of $n$ in $B_{k}$. As in the
real case, we note that
\[
\widetilde{R}_{n+1}-\widetilde{R}_{n}=\frac{2}{\sqrt{\pi}}\left(
\frac{\widetilde{D}^{k}-\widetilde{D}^{k-1}}{2^{k-1}}\right)  .\text{ }%
\]
Since $\widetilde{D}<2$ we have
\[
\frac{\widetilde{D}^{k}-\widetilde{D}^{k-1}}{2^{k-1}}>\frac{\widetilde
{D}^{k+1}-\widetilde{D}^{k}}{2^{k}}%
\]
and $\left(  \widetilde{R}_{n+1}-\widetilde{R}_{n}\right)  _{n=1}^{\infty}$ is
decreasing. Besides%
\[
\lim_{n\rightarrow\infty}\left(  \widetilde{R}_{n+1}-\widetilde{R}_{n}\right)
=0
\]
and%
\[
\widetilde{R}_{n+1}-\widetilde{R}_{n}<0.44\cdot n^{-0.695025}%
\]
since
\begin{align*}
\widetilde{R}_{n+1}-\widetilde{R}_{n}  &  \leq\left(  \frac{\widetilde{D}}%
{2}\right)  ^{k-1}\frac{2}{\sqrt{\pi}}\left(  \widetilde{D}-1\right)
\leq\left(  \frac{\widetilde{D}}{2}\right)  ^{\log_{2}\left(  \frac{n}%
{2}\right)  }\frac{2}{\sqrt{\pi}}\left(  \widetilde{D}-1\right) \\
&  \leq\left(  \frac{4}{\sqrt{\pi}}-\frac{4}{e^{\frac{1}{2}-\frac{1}{2}\gamma
}\sqrt{\pi}}\right)  n^{\log_{2}\left(  \frac{e^{\frac{1}{2}-\frac{1}{2}%
\gamma}}{2}\right)  }\\
&  <0.44\cdot n^{-0.695025}.
\end{align*}
Thus we get%
\begin{equation}
\widetilde{R}_{n}\leq1+\left(  \frac{4}{\sqrt{\pi}}-\frac{4}{e^{\frac{1}%
{2}-\frac{1}{2}\gamma}\sqrt{\pi}}\right)  {\textstyle\sum\limits_{j=1}^{n-1}%
}j^{^{\log_{2}\left(  \frac{e^{\frac{1}{2}-\frac{1}{2}\gamma}}{2}\right)  }}
\label{iiu2}%
\end{equation}
for all $n\geq2$. Numerically%
\[
K_{n}<\widetilde{R}_{n}<1+0.44\cdot{\sum\limits_{j=1}^{n-1}}j^{-0.695025}%
\]
for all $n\geq2$. Proceeding as in Section \ref{occ} we obtain%
\begin{equation}
K_{n}<\frac{\left(  \frac{4}{\sqrt{\pi}}-\frac{4}{e^{\frac{1}{2}-\frac{1}%
{2}\gamma}\sqrt{\pi}}\right)  }{1+\log_{2}\left(  \frac{e^{\frac{1}{2}%
-\frac{1}{2}\gamma}}{2}\right)  }\left(  n-1\right)  ^{\log_{2}\left(
e^{\frac{1}{2}-\frac{1}{2}\gamma}\right)  }+C_{0} \label{unb}%
\end{equation}
with%
\[
C_{0}=\left(  \frac{2e^{\frac{1}{2}}-2e^{\frac{1}{2}\gamma}}{\sqrt{\pi}%
}\right)  \left(  \frac{-4e^{\frac{1}{2}}\ln2+\left(  1-\gamma\right)  \left(
e^{\frac{1}{2}}+2e^{\frac{1}{2}\gamma}\right)  }{e^{\frac{1}{2}\gamma+\frac
{1}{2}}\left(  1-\gamma\right)  }\right)  +1
\]
for all $n\geq2$. Numerically
\[
K_{n}<1.41\left(  n-1\right)  ^{0.304975}-0.04
\]
for all $n\geq2.$

We recall that in \cite{ap2} it is shown that
\begin{equation}
K_{n}\nsim n^{r}\text{ for all }r>q:=\log_{2}\left(  e^{\frac{1}{2}-\frac{1}{2}\gamma}\right). \label{q1q1}
\end{equation}
We remark that the power of $\left(  n-1\right)  $ in \eqref{unb} is precisely the value of $q$ in \eqref{q1q1}, showing that in this case our \textquotedblleft perturbation argument\textquotedblright\ also does not cause any asymptotic damage.

Finally, using the same argument of the previous section, for any $\varepsilon>0$, we have
\[
K_{n+1}-K_{n}<\left(  \frac{4}{\sqrt{\pi}}-\frac{4}{e^{\frac{1}{2}-\frac{1}{2}\gamma}\sqrt{\pi}}\right)  n^{\log_{2}\left(  \frac{e^{\frac{1}{2}-\frac{1}{2}\gamma}}{2}\right)  +\varepsilon}
\]
for infinitely many $n$'s.

\section{The Kahane--Salem--Zygmund \& Bohnenblust--Hille inequalities}

The Kahane--Salem--Zygmund inequality~(see \cite[Theorem 4, Chapter 6]{Kah}
and also \cite{sal}) is a powerful result which has been useful for several
applications (see \cite{Bay, football, boasK, Kor, Pit}). This inequality, in
its whole generality, is a probabilistic result but in our case (and
apparently in most of the applications) a weaker version is enough:

\begin{theorem}
[Kahane--Salem--Zygmund inequality]Let $m,n$ be positive integers. Then there
are signs $\varepsilon_{\alpha}=\pm1$ so that the $m$-homogeneous polynomial
\[
P_{m,n}:\ell_{\infty}^{n}\rightarrow\mathbb{C}
\]
given by
\[
P_{m,n}(z)={\textstyle\sum_{\left\vert \alpha\right\vert =m}}\varepsilon_{\alpha}z^{\alpha}
\]
satisfies
\[
\left\Vert P_{m,n}\right\Vert \leq Cn^{\left(  m+1\right)  /2}\sqrt{\log m}
\]
where $C>0$ is an universal constant (it does not depend on $n$ or $m$).
\end{theorem}

Connections between the Bohnenblust--Hille inequality and the Kahane--Salem--Zygmund inequality are known; our aim is to stress even closer connections that may be useful in future investigations.

In the next subsection we show that the optimality of the exponent $\frac{2m}{m+1}$ is a simple corollary of the Kahane--Salem--Zygmund inequality.

\subsection{A new (and simple) proof that the power $\frac{2m}{m+1}$ is sharp}

As mentioned in the Introduction, there seems to exist no direct proof of the optimality of the exponent $\frac{2m}{m+1}$ in the Bohnenblust--Hille inequalities.

In \cite{Blei} there is an alternative proof for the case of multilinear mappings, but the arguments are also highly nontrivial, involving $p$-Sidon sets and sub-Gaussian systems. Here we shall show that the optimality of the power $\frac{2m}{m+1}$ is a straightforward consequence of the Kahane--Salem--Zygmund inequality (the results are stated for complex scalars but the same argument holds for real scalars, since it is obvious that the Kahane--Salem--Zygmund inequality can be adapted to the case of real scalars). It is worth mentioning that our proof in fact proves more than the statement of the theorem (see Theorem \ref{novo00} below).

\begin{theorem}\label{BMM} 
The power $\frac{2m}{m+1}$ in the Bohnenblust--Hille inequalities is sharp.
\end{theorem}

\begin{proof}
Let $m\geq2$ be a fixed positive integer. For each $n$, let $P_{m,n}:\ell_{\infty}^{n}\rightarrow\mathbb{C}$ be the $m$-homogeneous polynomial satisfying the Kahane--Salem--Zygmund inequality. For our goals it suffices to deal with the case $n>m.$

Let $q<\frac{2m}{m+1}$. Then a simple combinatorial calculation shows that
\[
{\textstyle\sum_{\left\vert \alpha\right\vert =m}}\left\vert \varepsilon
_{\alpha}\right\vert ^{q}=p(n)+\frac{1}{m!}{\textstyle\prod\limits_{k=0}%
^{m-1}}(n-k),
\]
where $p\left(  n\right)  >0$ is a polynomial of degree $m-1.$ If the
polynomial Bohnenblust--Hille inequality was true with the power $q$, then
there would exist a constant $C_{m,q}>0$ so that
\[
C_{m,q}C\geq\frac{1}{n^{\left(  m+1\right)  /2}\sqrt{\log m}}\left(
p(n)+\frac{1}{m!}{\textstyle\prod\limits_{k=0}^{m-1}}(n-k)\right)  ^{1/q}%
\]
for all $n$. If we raise both sides to the power of $q$ and let $n\rightarrow
\infty$ we obtain
\[
\left(  C_{m,q}C\right)  ^{q}\geq\lim_{n\rightarrow\infty}\left(  \frac
{r(n)}{m!n^{q\left(  m+1\right)  /2}\left(  \sqrt{\log m}\right)  ^{q}}%
+\frac{p(n)}{n^{q\left(  m+1\right)  /2}\left(  \sqrt{\log m}\right)  ^{q}%
}\right)  ,
\]
with
\[
r(n)={\textstyle\prod\limits_{k=0}^{m-1}}(n-k).
\]
Since
\[
\deg r=m>\frac{q(m+1)}{2}%
\]
we have
\[
\lim_{n\rightarrow\infty}\left(  \frac{r(n)}{m!n^{q\left(  m+1\right)
/2}\left(  \sqrt{\log m}\right)  ^{q}}+\frac{p(n)}{n^{q\left(  m+1\right)
/2}\left(  \sqrt{\log m}\right)  ^{q}}\right)  =\infty,
\]
a contradiction. Since the multilinear Bohnenblust--Hille inequality (with a
power $q$) implies the polynomial Bohnenblust--Hille inequality with the same
power, we conclude that $\frac{2m}{m+1}$ is also sharp in the multilinear case.
\end{proof}

From now on a Bernoulli polynomial is a polynomial whose coefficients are $1$
or $-1.$ Note that our proof, albeit elementary, it proves in fact a stronger
(although probably known) result:

\begin{theorem}
\label{novo00} Let $q\geq1$ be so that there is a constant $C_{q,m}\geq1$ such
that
\[
\left(  {\textstyle\sum\limits_{\left\vert \alpha\right\vert =m}}\left\vert
\varepsilon_{\alpha}\right\vert ^{q}\right)  ^{\frac{1}{q}}\leq C_{q,m}%
\left\Vert P_{m,n}\right\Vert
\]
for every $m$-homogeneous Bernoulli polynomial $P_{m,n}:\ell_{\infty}%
^{n}\rightarrow\mathbb{C}$,%
\[
P_{m,n}(z)={\sum_{\left\vert \alpha\right\vert =m}}\varepsilon_{\alpha
}z^{\alpha}.
\]
Then%
\[
q\geq\frac{2m}{m+1}.
\]

\end{theorem}

\subsection{The Kahane--Salem--Zygmund constant \& the op\-ti\-mal (complex)
polynomial Bo\-hnen\-blust--Hille cons\-tants}

From now on $K_{m}^{pol}$ denotes the optimal constant satisfying the
($m$-homogeneous) polynomial Bohnenblust--Hille inequality (complex case) and
$C$ denotes the universal constant from the Kahane--Salem--Zygmund inequality.

In this subsection we sketch some connections between the universal constant
from the Kahane--Salem--Zygmund inequality and the optimal constants from the
(complex) polynomial Bohnenblust--Hille inequality; this approach may be
useful to build strategies (or at least to show that some strategies are not
adequate) for the investigation of lower bounds for the complex polynomial
Bohnenblust--Hille constants.

The search of the optimal constants of any nature is naturally divided in two
different approaches: the search of upper estimates and lower estimates. For
the polynomial Bohnenblust--Hille inequalities the situation is not different.

The best result on upper bounds for the (complex) polynomial
Bohnenblust--Hille constants is due to Defant \textit{et al.}, published in
2011 in \cite{annals} (see \eqref{defant55}). On the other hand, the search
for lower bounds presents very few advances in the complex case. Up to now the
unique nontrivial result (for complex scalars) in this direction states that
\[
K_{2}^{pol}\geq1.1066\text{ (\cite{munn}).}%
\]
Let us begin with a simple remark: the optimal constants satisfying the
($m$-homogeneous) polynomial Bohnenblust--Hille inequality can be used to
estimate the universal constant $C$ from the Kahane--Salem--Zygmund inequality.

In fact, using the same procedure of the previous subsection (choosing $m=n$)
we conclude that
\begin{equation}
K_{m}^{pol}C\geq\frac{1}{m^{\frac{m+1}{2}}\sqrt{\log m}}\left(  \frac{\left(
2m-1\right)  !}{m!\left(  m-1\right)  !}\right)  ^{\frac{m+1}{2m}}
\label{der3}%
\end{equation}
for all $m\geq2.$ A rapid calculation gives us a lower bound for the optimal value of $C.$ In fact, for $m=2$ in \eqref{der3}, we have
\begin{equation}
K_{2}^{pol}C>0.9680. \label{der2}%
\end{equation}

But, from \cite[Th III.1]{Q} (in this case the estimate from \cite[Th
III.1]{Q} is better than \eqref{defant55}) we know that%
\[
K_{2}^{pol}\leq1.7432
\]
and thus we conclude that%
\begin{equation}
C>\frac{0.9680}{1.7432}>0.5553. \label{der}%
\end{equation}
However, using a different technique (in fact, using exhaustion for $m=n=2$)
we can obtain a quite better estimate for $C$. From \cite[eq 3.1]{aron}, and
the Maximum Modulus Principle (as used in \cite{munn}) we can show that if
$P_{2}:\ell_{\infty}^{2}\rightarrow\mathbb{C}$ is defined by
\begin{equation}
P_{2}(z_{1},z_{2})=az_{1}^{2}+bz_{2}^{2}+cz_{1}z_{2} \label{sqw}%
\end{equation}
with $a,b,c\in\mathbb{R}$, then%
\[
\Vert P_{2}\Vert=%
\begin{cases}
|a+b|+|c| & \text{if $ab\geq0$ or $|c(a+b)|>4|ab|$,}\\
\left(  |a|+|b|\right)  \sqrt{1+\frac{c^{2}}{4|ab|}} & \text{otherwise.}%
\end{cases}
\]
So if $a,b,c\in\left\{  -1,1\right\}  $ the possible norms of $P_{2}$ are $3$
and $\sqrt{5}.$ So, it is immediate that%
\begin{equation}
C\geq\frac{\sqrt{5}}{2^{3/2}\sqrt{\log2}}>0.9495. \label{der33}%
\end{equation}

We have not found lower estimates for $C$ in the literature; although probably \eqref{der33} could be of interest. The gap between the estimates \eqref{der} and \eqref{der33} is probably due the fact that Bernoulli polynomials seem to
be not good candidates for furnishing lower bounds for $K_{m}^{pol}$. In fact,
in \cite{munn} the best choice (for obtaining lower bounds for the polynomial
Bohnenblust--Hille constant for $2$-homogeneous polynomials) over all
polynomials of the form \eqref{sqw} was
\[
P_{2}(z_{1},z_{2})=z_{1}^{2}-z_{2}^{2}+\frac{352\,203}{125\,000}z_{1}z_{2}.
\]
Since $C$ is an universal constant, it is presumable that \eqref{der3} may be
not useful for estimating the Bohnenblust--Hille constants. For a stronger
version of \eqref{der3} it seems that we should avoid the use of the universal
constant $C$ and use particular values of $C$ for specific values of $m,n.$
More precisely, if $n\geq m\geq2$ are fixed, the Kahane--Salem--Zygmund
inequality tells us that there is a constant $C_{m,n}$ with $0<C_{m,n}\leq C$
so that that there are signs $\varepsilon_{\alpha}=\pm1$ and an $m$%
-homogeneous polynomial
\begin{align*}
P_{m,n}  &  :\ell_{\infty}^{n}\rightarrow\mathbb{C}\\
P_{m,n}(z)  &  ={\sum_{\left\vert \alpha\right\vert =m}}\varepsilon_{\alpha
}z^{\alpha},
\end{align*}
with
\[
\left\Vert P_{m,n}\right\Vert \leq C_{m,n}n^{\left(  m+1\right)  /2}\sqrt{\log
m}.
\]
Keeping this notation we have that
\[
K_{m}^{pol}C_{m,n}\geq\frac{1}{n^{\frac{m+1}{2}}\sqrt{\log m}}\left(
\frac{\left(  n+m-1\right)  !}{m!\left(  n-1\right)  !}\right)  ^{\frac
{m+1}{2m}}%
\]
whenever $m,n$ are positive integers with $n\geq m\geq2$. So, the search of
the optimal values of $C_{m,n},$ besides its intrinsic interest, may help in
the incipient investigation of lower bounds for the optimal Bohnenblust--Hille
constants $K_{m}^{pol}$. However, our suspicion is that Bernoulli polynomials
(and thus the Kahane-Salem--Zygmund inequality) are effective exclusively for
the proof of the optimality of the exponent $\frac{2m}{m+1}$ (Theorem
\ref{BMM}) and, as it happened in the case $m=n=2$, they seem not efficient
for the estimation of the constants $K_{m}^{pol}$.

\subsection{The Kahane--Salem--Zygmund inequality: is the power $\frac{m+1}%
{2}$ optimal even for real scalars?}

It is obvious that the norm of a Bernoulli polynomial over the complex scalar
field is never smaller than its norm over the real scalar field. More
precisely, if $\mathbb{K}=\mathbb{R}$ or $\mathbb{C}$, $\varepsilon_{\alpha}
\in\left\{  -1,1\right\}  $ and
\begin{align*}
P_{\mathbb{K}}  &  :\ell_{\infty}^{n}(\mathbb{K})\rightarrow\mathbb{K}\\
P_{\mathbb{K}}(z)  &  ={\sum_{\left\vert \alpha\right\vert =m}}\varepsilon
_{\alpha}z^{\alpha},
\end{align*}
then
\[
\left\Vert P_{\mathbb{R}}\right\Vert \leq\left\Vert P_{\mathbb{C}}\right\Vert
.
\]
A concrete example: if $P_{\mathbb{K}}:\ell_{\infty}^{2}(\mathbb{K}
)\rightarrow\mathbb{K}$ is given by
\[
P_{\mathbb{K}}(z)=z_{1}^{2}-z_{2}^{2}+z_{1}z_{2},
\]
then
\[
\left\Vert P_{\mathbb{R}}\right\Vert =\frac{5}{4}<\sqrt{5}=\left\Vert
P_{\mathbb{C}}\right\Vert .
\]
So, as mentioned in Subsection 10.1, it is obvious that the
Kahane--Salem--Zygmund inequality holds for real scalars. It seems to be
well-known that the power $\frac{m+1}{2}$ in the Kahane--Salem--Zygmund
inequality is optimal (for complex scalars) but for real scalars the result
seems to be not clear. In any case, the following straightforward proof (via
Bohnenblust--Hille inequality) that the exponent $\frac{m+1}{2}$ is optimal
for both real or complex scalars seems to be of independent interest.

\begin{theorem}
The power $\frac{m+1}{2}$ in the Kahane--Salem--Zygmund inequality is optimal
for both real and complex scalars.
\end{theorem}

\begin{proof}
The argument is similar to the proof of the optimality of Theorem \ref{BMM}.
Let $m\geq2$ be a fixed positive integer, $n\geq m$ and $\mathbb{K}%
=\mathbb{R}$ or $\mathbb{C}$. Let us suppose that the Kahane--Salem--Zygmund
inequality is valid for an exponent $q<\frac{m+1}{2}.$ For each $n$ and $m$
let $P_{m,n}:\ell_{\infty}^{n}(\mathbb{K})\rightarrow\mathbb{K}$ be the
$m$-homogeneous polynomial satisfying the Kahane--Salem--Zygmund inequality
with this exponent $q$. As in the proof of Theorem \ref{BMM}, we have
\[
{\textstyle\sum_{\left\vert \alpha\right\vert =m}}\left\vert \varepsilon
_{\alpha}\right\vert ^{\frac{2m}{m+1}}=p(n)+\frac{1}{m!}{\textstyle\prod
\limits_{k=0}^{m-1}}(n-k),
\]
where $p\left(  n\right)  >0$ is a polynomial of degree $m-1$; and there would
exist a constant $C_{(q)}>0$ so that
\[
K_{m}^{pol}C_{(q)}\geq\frac{1}{n^{q}\sqrt{\log m}}\left(  p(n)+\frac{1}%
{m!}{\textstyle\prod\limits_{k=0}^{m-1}}(n-k)\right)  ^{\left(  m+1\right)
/2m}%
\]
for all $n.$ Hence
\[
\left(  K_{m}^{pol}C_{(q)}\right)  ^{\frac{2m}{m+1}}\geq\lim_{n\rightarrow
\infty}\left(  \frac{r(n)}{m!n^{\frac{2mq}{m+1}}\left(  \sqrt{\log m}\right)
^{\frac{2m}{m+1}}}+\frac{p(n)}{n^{\frac{2mq}{m+1}}\left(  \sqrt{\log
m}\right)  ^{\frac{2m}{m+1}}}\right)  .
\]
with $r$ as in the proof of Theorem \ref{BMM}. Since
\[
\deg r=m>\frac{2mq}{m+1}%
\]
we obtain a contradiction.
\end{proof}

\section{Is there a strong multilinear Bohnenblust--Hille inequality?}

Of course, there are still a lot of open questions related to the growth of the optimal constants satisfying the multilinear (and polynomial) Bohnenblust--Hille inequalities to be solved. For example, it is {\em not} clear that the optimal constants $\left(  K_{n}\right)  _{n=1}^{\infty}$ satisfying the multilinear Bohnenblust--Hille inequality grow to infinity. It seems that the original estimates induce us to think that in fact $K_{n}\rightarrow\infty,$ but it purports to exist no other evidence for this.

Although there still remains in a veil of mystery, combining all the information obtained thus far we believe that the possibility of boundedness of the constants of the multilinear Bohnenblust--Hille inequality should be seriously considered. We prefer not to conjecture that it is true, but instead we pose it as an open problem:

\begin{problem}
Is there an universal constant $K_{\mathbb{K}}$ so that
\[
\left(  \sum\limits_{i_{1},\ldots,i_{m}=1}^{N}\left\vert U(e_{i_{^{1}}},\ldots,e_{i_{m}})\right\vert ^{\frac{2m}{m+1}}\right)  ^{\frac{m+1}{2m}}\leq K_{\mathbb{K}}\sup_{z_{1},\ldots,z_{m}\in\mathbb{D}^{N}}\left\vert U(z_{1},\ldots,z_{m})\right\vert
\]
for every positive integer $m\geq1,$ all $m$-linear forms $U:\mathbb{K}^{N}\times\cdots\times\mathbb{K}^{N}\rightarrow\mathbb{K}$ and every positive integer $N$?
\end{problem}

\begin{conjecture}
If the answer to the previous problem is positive, we conjecture that $K_{\mathbb{R}}=2$ and $K_{\mathbb{C}}\leq 2.$
\end{conjecture}

We justify our conjecture that $K_{\mathbb{R}}=2$ motivated by the lower bounds obtained in \cite{diniz2} for the constants of the multilinear Bohnenblust--Hille inequality (real case),

\begin{equation}
K_{m}\geq2^{1-\frac{1}{m}}. \label{lbb}%
\end{equation}
We stress that the case $m=2$ in \eqref{lbb} is sharp, i.e., $\sqrt{2}$ is the optimal constant for the $2$-linear Bohnenblust--Hille inequality (real case). As a matter of fact, if we consider $m=1$, then the formula \eqref{lbb} also provides a sharp value. So, since in each level $m$, the lower estimate for $K_{m}$ is obtained by the same induction argument (for details, see \cite{diniz2}) and since the cases $m=1,2$ provide sharp constants, we believe that it is not impossible that the formula \eqref{lbb} gives the exact
constants for the Bohnenblust--Hille constants. We reinforce our belief by observing the several recent works showing that the growth of the constants in the Bohnenblust--Hille inequality is it in fact quite slower than the original estimates had predicted.

It seems to be {\em folklore} (although not formally proved) that the constants for the case of real scalars are bigger than the constants for the complex case. For example, for $m=2$ one has $K_{2}=\sqrt{2}$ in the real case and $K_{2}\leq\frac{2}{\sqrt{\pi}}<\sqrt{2}$ in the complex case. Besides, the growth of the constants in the complex case seems to be slower than the growth
in the real case (see \cite{ap, ap2}). So, if our conjecture is correct, it seems natural to think that $K_{\mathbb{C}}\leq K_{\mathbb{R}}$.

It is our belief that the possibility of a ``\emph{strong Bohnenblust--Hille inequality}'' only applies to \emph{multilinear mappings} since, in the case of polynomials, it is essentially shown in \cite{munn} that (at least for real scalars) the optimal constants are not bounded.

\section{Appendix}

Very recently, explicit applications on Quantum Information Theory (more precisely quantum XOR games) of the low growth of the multilinear Bohnenblust--Hille constants (case of real scalars) were provided by A. Montanaro (see \cite{monta}). In view of this new panorama we think that it is worth mentioning that the techniques used in the present paper can be adapted to a wide range of parameters. More precisely, using our techniques we can estimate the constants satisfying Bohnenblust--Hille type inequalities when $\frac{2n}{n+1}$ is replaced by any $q\in\left[  \frac{2n}{n+1},\infty\right)$. Since, for $q>2$, the constants are equal to $1$, the  nontrivial cases take place for $q\in\left[  \frac{2n}{n+1},2\right)$ .

The case of Littlewood's $\frac{4}{3}$ inequality was recently explored in \cite{ap2} and the estimates of $L_{\mathbb{K},r}$ satisfying
\begin{equation}
\left(  \sum\limits_{i,j=1}^{N}\left\vert U(e_{i},e_{j})\right\vert^{r}\right)^{\frac{1}{r}}\leq L_{\mathbb{K},r}\left\Vert U\right\Vert
\label{p9a}
\end{equation}
were obtained. More precisely, it was shown that, for all $r\in\left[\frac{4}{3},t_{0}\right]  ,$ with $t_{0}\approx 1.92068$,
\[
L_{\mathbb{R},r}=\left\{
\begin{array}[c]{c}
2^{\frac{2-r}{r}}\text{ for all }r\in\left[  \frac{4}{3},t_{0}\right]  \\
1\text{ for all }r\geq2,
\end{array}
\right.
\]
and, for $t_{0}<r<2,$ we have
\[
2^{\frac{2-r}{r}}\leq L_{\mathbb{R},r}\leq\frac{1}{\sqrt{2}}\left(\frac{\sqrt{\pi}}{\Gamma\left(  \frac{4+r}{2\left(  4-r\right)  }\right)
}\right)^{\left(  4-r\right)/2r}.
\]
For each $t\in\lbrack1,2)$, let
\[
E_{t,n}=\frac{2nt}{\left(n-1\right) t+2}
\]
for all $n\in%
\mathbb{N}
.$ Note that
\[
E_{1,n}=\left(  \frac{2n}{n+1}\right)  _{n\in\mathbb{N}}
\]
is the \textquotedblleft Bohnenblust--Hille sequences of
exponents\textquotedblright. Note also that for each $t\in(1,2),$ we have
\[
\frac{2nt}{\left(  n-1\right)  t+2}>\left(  \frac{2n}{n+1}\right)  .
\]

Thus, there exist a $C_{n,t}\geq1$ so that
\begin{equation}
\left(  \sum_{i_{1},\ldots, i_{n}}^{N}\left\vert U\left(  e_{i_{1}},\ldots, e_{i_{n}}\right)  \right\vert ^{\frac{2nt}{\left(  n-1\right)  t+2}}\right)
^{\frac{\left(  n-1\right)  t+2}{2nt}}\leq C_{n,t}^{\mathbb{K}}\left\Vert
U\right\Vert ,\label{desia}
\end{equation}
for all $n$-linear forms $U:{\ell_{\infty}^{N}\times \cdots \times \ell_{\infty}^{N}}\rightarrow\mathbb{K},$ and positive integer $N$, with $\mathbb{K=\mathbb{R}}$ or $\mathbb{C}$.

\bigskip

\noindent In what follows we shall show the \textit{continuum version} of the results of the present paper.

\subsection{Estimates in the case $\mathbb{R}$.}
The following result can be proved following the lines of \cite{jmaa}

\begin{theorem}
\label{REALa}If $t\in\lbrack1,2),$ then
\begin{equation}
C_{n,t}^{\mathbb{R}}=\left\{
\begin{array}[c]{ll}
1 & \text{ if } n=1,\\
\left(  A_{\frac{2nt}{\left(  n-2\right)  t+4}}^{-\frac{n}{2}}C_{\frac{n}{2},t}^{\mathbb{R}}\right)  & \text{ if } n \text{ is even, and}\\
\left(  A_{\frac{2\left(  n-1\right)  t}{\left(  n-3\right)  t+4}}^{-\left(
n+1\right)  /2}C_{\frac{n-1}{2},t}^{\mathbb{R}}\right)  ^{\frac{n-1}{2n}}\left(  A_{\frac{2\left(  n+1\right)  t}{\left(
n-1\right)  t+4}}^{-\left(  n-1\right)  /2}C_{\frac{n+1}{2},t}^{\mathbb{R}}\right)  ^{\frac{n+1}{2n}} & \text{ if } n \text{ is odd.}
\end{array}
\right.\label{origa}
\end{equation}
\end{theorem}

Below, we state how the {\em continuum} versions of our results apply to the case of real scalars (always, when $t=1$, we recover the respective original result for the Bohnenblust--Hille inequality):

\begin{theorem}[The Fundamental Lemma - continuum version]\label{t4c}
For each $t\in \lbrack1,2),$ there is a sequence satisfying \eqref{desia} and so that $\left(R_{n+1,t}-R_{n,t}\right)_{n=1}^{\infty}$ is decreasing and converges to zero. Moreover
\begin{equation}
R_{n+1,t}-R_{n,t}\leq\left(  2^{\frac{3}{2}}-2^{\frac{t+1}{t}}e^{\frac{t-2}{t}+\frac{\left(  2-t\right)  \gamma}{2t}}\right)  n^{\log_{2}\left(
2^{\frac{-t-2}{2t}}e^{\frac{2-t}{t}-\frac{\left(  2-t\right)  \gamma}{2t}}\right)  }\label{rhsa}
\end{equation}
for all $n \in \mathbb{N}$.
\end{theorem}

\begin{theorem}\label{56a}
For each $t\in\lbrack1,2),$ let $\left(  K_{n,t}^{\mathbb{R}}\right)  _{n=1}^{\infty}$ be the sequence of the optimal constants satisfying \eqref{desia}. If there is a constant $M_{t}\in\lbrack-\infty,\infty]$ so that
\[
\lim_{n\rightarrow\infty}\left(  K_{n+1,t}^{\mathbb{R}}-K_{n,t}^{\mathbb{R}}\right)  =M_{t}
\]
then $M_{t}=0$.
\end{theorem}

\begin{theorem}\label{utya}
For each $t\in\lbrack1,2),$ let $\left(  K_{n,t}^{\mathbb{R}}\right)_{n=1}^{\infty}$ be the sequence of the optimal constants satisfying \eqref{desia}. For any $\varepsilon>0$, we have
\begin{equation}
K_{n+1,t}^{\mathbb{R}}-K_{n,t}^{\mathbb{R}}<\left(  2^{\frac{3}{2}}-2^{\frac{t+1}{t}}e^{\frac{t-2}{t}+\frac{\left(
2-t\right)  \gamma}{2t}}\right)  n^{\log_{2}\left(  2^{\frac{-t-2}{2t}}e^{\frac{2-t}{t}-\frac{\left(  2-t\right)  \gamma}{2t}}\right)  +\varepsilon}\label{qza}
\end{equation}
for infinitely many $n \in \mathbb{N}$.
\end{theorem}

\begin{theorem}
\label{tt55a}For each $t\in\lbrack1,2)$, the optimal constants satisfying \eqref{desia} are so that
\[
K_{n,t}^{\mathbb{R}}<1+\left(  2^{\frac{3}{2}}-2^{\frac{t+1}{t}}e^{\frac{t-2}{t}+\frac{\left(
2-t\right)  \gamma}{2t}}\right)  {\textstyle\sum\limits_{j=1}^{n-1}}j^{\log_{2}\left(  2^{\frac{-t-2}{2t}}e^{\frac{2-t}{t}-\frac{\left(
2-t\right)  \gamma}{2t}}\right)  }.
\]
for every $n\geq2$.
\end{theorem}

\begin{corollary}\label{nnhha}
For each $t\in\lbrack1,2)$, the optimal constants satisfying \eqref{desia} are so that
\begin{equation}
K_{n,t}^{\mathbb{R}}<c\left(  t\right)  \left(  n-1\right)  ^{r\left(  t\right)  }+p\left(t\right),\label{999a}
\end{equation}
where (see Figure \ref{pcr}):
\begin{itemize}
\item[] \hspace{-1.25cm}$\displaystyle p(t)=1-\left( 2^{3/2}-2^{\frac{t+1}{t}}e^{\frac{t-2}{t}+\frac{\left(
2-t\right) \gamma }{2t}}\right) \left( \frac{2^{\frac{3t-2}{2t}}te^{\frac{2-t}{t}-\frac{\left( 2-t\right) \gamma }{2t}}}{t-2+2t\log _{2}e^{\frac{2-t}{t}-\frac{\left( 2-t\right) \gamma }{2t}}}-1-2^{\frac{-t-2}{2t}}e^{\frac{2-t}{t}-\frac{\left( 2-t\right) \gamma }{2t}}\right)$,
\vspace{.25cm}
\item[] \hspace{-1.25cm}$\displaystyle c(t) =\frac{4t\left(  \sqrt{2}-2^{\frac{1}{t}}e^{\frac{t-2}{t}+\frac{\left(  2-t\right)  \gamma}{2t}}\right)  }{t-2+2t\log_{2}e^{\frac{2-t}{t}-\frac{\left(  2-t\right)  \gamma}{2t}}}$, and
\vspace{.25cm}
\item[] \hspace{-1.25cm}$\displaystyle r(t)  =\frac{t-2}{2t}+\log_{2}e^{\frac{2-t}{t}-\frac{\left(2-t\right)  \gamma}{2t}}$.
\end{itemize}
\end{corollary}
\begin{figure}
\centering
\begin{tabular}{cc}
\includegraphics[width=0.4\textwidth]{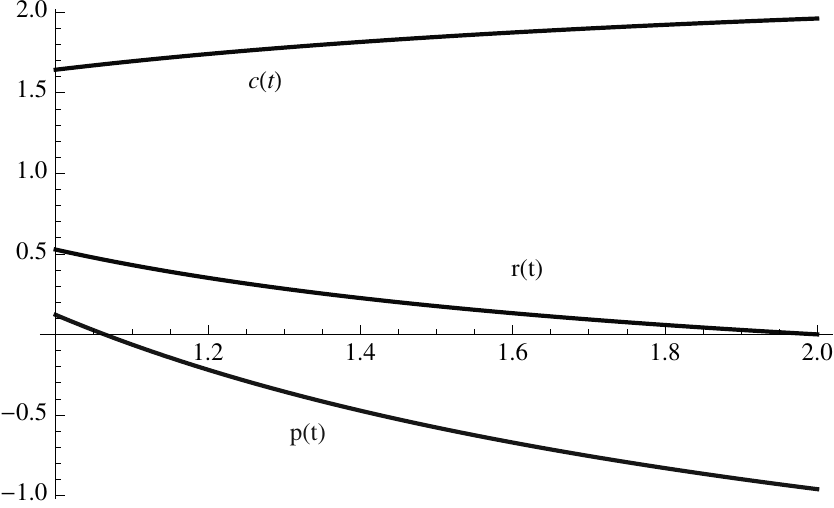} & \includegraphics[width=0.4\textwidth]{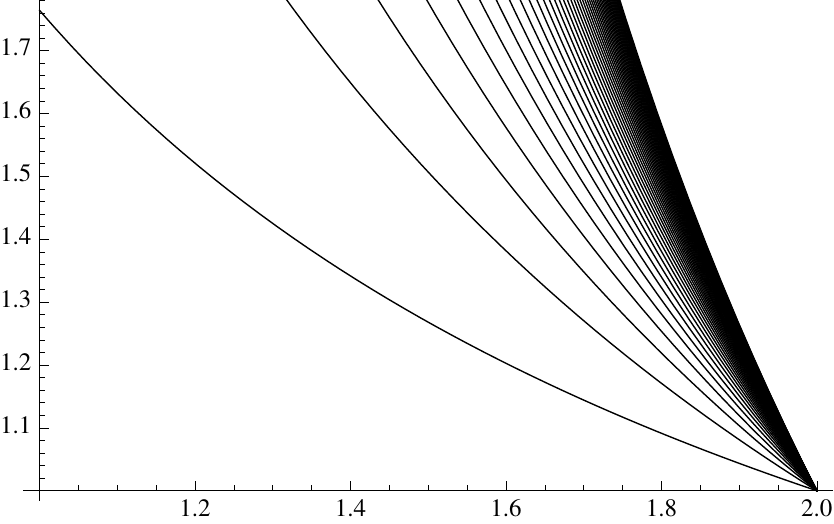}\\
(a) Plots of $p(t), c(t)$, and $r(t)$ for $t \in [1,2)$. & (b) $K_{n,t}^{\mathbb{R}}$ for $2 \le n \le 50$.
\end{tabular}
\caption{}\label{pcr}
\end{figure}

By using a very recent technique from \cite{diniz2,ap2} we can also prove the following.
\begin{theorem}
\label{lowereal}For all $t\in\lbrack1,2)$ and $n\in\mathbb{N}$, we have
\[
K_{n,t}^{\mathbb{R}}\geq2^{\frac{\left(  n-1\right)  \left(  2-t\right)}{nt}}.
\]
\end{theorem}

\subsection{Estimates in the case $\mathbb{C}$}

\begin{theorem}\label{Comp}
Let $t\in\lbrack1,2)$. For every $n\in\mathbb{N}$ and $X_{1},\ldots, X_{n}$ Banach spaces over $\mathbb{C}$,
\[
\Pi_{\left(  \frac{2nt}{\left(  n-1\right)  t+2};1\right)  }\left( X_{1},\ldots, X_{n};\mathbb{C}\right)  =\mathcal{L}\left(  X_{1},\ldots, X_{n};\mathbb{C}\right)  \text{ e }\left\Vert .\right\Vert _{\pi\left(  \frac{2nt}{\left(  n-1\right)  t+2};1\right)  }\leq C_{n,t}^{\mathbb{C}}\left\Vert .\right\Vert
\]
with
\[
C_{n,t}^{\mathbb{C}
}=\left\{
\begin{array}[c]{ll}
1 & \text{ if } n=1,\\
\\
\frac{C_{\frac{n}{2},t}^{\mathbb{C}}}{\left(  \widetilde{A_{\frac{2nt}{\left(  n-2\right)  t+4}}}\right)  ^{n/2}
} & \text{ if } n\,\ \text{is even, and }\\
\\
\left(  \frac{C_{\frac{n-1}{2},t}^{\mathbb{C}}}{\left(  \widetilde{A_{\frac{2\left(  n-1\right)  t}{\left(  n-3\right)
t+4}}}\right)  ^{\left(  n+1\right)  /2}}\right)  ^{\frac{n-1}{2n}}\left(
\frac{C_{\frac{n+1}{2},t}^{\mathbb{C}}}{\left(  \widetilde{A_{\frac{2\left(  n+1\right)  t}{\left(  n-1\right)
t+4}}}\right)  ^{\left(  n-1\right)  /2}}\right)  ^{\frac{n+1}{2n}} & \text{if } n \text{ is odd.}
\end{array}
\right.
\]
\end{theorem}

\begin{theorem}
[The Fundamental Lemma - continuum version complex]For each $t\in\lbrack1,2),$
there is a sequence satisfying \eqref{desia} and so that $\left(
\widetilde{R}_{n+1,t}-\widetilde{R}_{n,t}\right)  _{n=1}^{\infty}$ is
decreasing and converges to zero. Moreover
\[
\widetilde{R}_{n+1,t}-\widetilde{R}_{n,t}\leq\left(  \frac{2\left(\Gamma\left(  \frac{t+2}{2}\right)  \right)  ^{\frac{-1}{t}}\left(  e^{\left(
\frac{1}{4t}\left(  \gamma-1\right)  \left(  2t-4\right)  \right)  }-1\right)}{e^{\left(  \frac{1}{4t}\left(  \gamma-1\right)  \left(  2t-4\right)
\right)  }}\right)  n^{\log_{2}\left(  \frac{e^{\left(  \frac{1}{4t}\left(\gamma-1\right)  \left(  2t-4\right)  \right)  }}{2}\right)}
\]
for every $n \in \mathbb{N}$.
\end{theorem}

\begin{theorem}
For each $t\in\lbrack1,2),$ let $\left(  K_{n,t}^{\mathbb{C}}\right)  _{n=1}^{\infty}$ be the sequence of the optimal constants satisfying \eqref{desia}. If there is a constant $M_{t}\in\lbrack-\infty,\infty]$ so that
\[
\lim_{n\rightarrow\infty}\left(  K_{n+1,t}^{\mathbb{C}}-K_{n,t}^{\mathbb{C}}\right)  =M_{t}
\]
then $M_{t}=0$.
\end{theorem}

\begin{theorem}
For each $t\in\lbrack1,2),$ let $\left(  K_{n,t}^{\mathbb{C}}\right)  _{n=1}^{\infty}$ be the sequence of the optimal constants satisfying \eqref{desia}. For any $\varepsilon>0$, we have
\[
K_{n+1,t}^{\mathbb{C}}-K_{n,t}^{\mathbb{C}}<\left(  \frac{2\left(  \Gamma\left(  \frac{t+2}{2}\right)  \right)^{\frac{-1}{t}}\left(  e^{\left(  \frac{1}{4t}\left(  \gamma-1\right)  \left(
2t-4\right)  \right)  }-1\right)  }{e^{\left(  \frac{1}{4t}\left(\gamma-1\right)  \left(  2t-4\right)  \right)  }}\right)  n^{\log_{2}\left(\frac{e^{\left(  \frac{1}{4t}\left(  \gamma-1\right)  \left(  2t-4\right)\right)  }}{2}\right)  +\varepsilon}
\]
for infinitely many $n \in \mathbb{N}$.
\end{theorem}

\begin{theorem}
For each $t\in\lbrack1,2)$, the optimal constants satisfying \eqref{desia} are so that
\[
K_{n,t}^{\mathbb{C}}<1+\frac{2\left(  \Gamma\left(  \frac{t+2}{2}\right)  \right)  ^{\frac{-1}{t}}\left(  e^{\left(  \frac{1}{4t}\left(  \gamma-1\right)  \left(
2t-4\right)  \right)  }-1\right)  }{e^{\left(  \frac{1}{4t}\left(\gamma-1\right)  \left(  2t-4\right)  \right)  }}{\textstyle\sum
\limits_{j=1}^{n-1}}j^{\log_{2}\frac{e^{\left(  \frac{1}{4t}\left(\gamma-1\right)  \left(  2t-4\right)  \right)}}{2}}.
\]
for every $n\geq2$.
\end{theorem}

\begin{corollary}
For each $t\in\lbrack1,2)$, the optimal constants satisfying \eqref{desia} are so that
\[
K_{n,t}^{\mathbb{C}}<c^{\prime}\left(  t\right)  \left(  n-1\right)^{r^{\prime}(t)}+p^{\prime}\left(  t\right)  ,
\]
where (see Figure \ref{pcrx}):

\begin{itemize}
\item[] $\displaystyle p^{\prime}\left(  t\right)  =1+\frac{\left(  \frac{-2}{\log_{2}e^{\left(
\frac{1}{4t}\left(  \gamma-1\right)  \left(  2t-4\right)  \right)  }}+\frac
{2}{e^{\left(  \frac{1}{4t}\left(  \gamma-1\right)  \left(  2t-4\right)
\right)  }}+1\right)  }{\left(  \Gamma\left(  \frac{t+2}{2}\right)  \right)
^{\frac{1}{t}}\left(  e^{\left(  \frac{\left(  \gamma-1\right)  \left(
2t-4\right)  }{4t}\right)  }-1\right)  ^{-1}}$,
\vspace{.25cm}
\item[] $\displaystyle c^{\prime}\left(  t\right)  =\frac{2\left(  \Gamma\left(  \frac{t+2}{2}\right)  \right)  ^{\frac{-1}{t}}\left(  e^{\left(  \frac{\left(\gamma-1\right)  \left(  2t-4\right)  }{4t}\right)  }-1\right)  }{\left(\log_{2}e^{\left(  \frac{\left(  \gamma-1\right)  \left(  2t-4\right)  }{4t}\right)  }\right)  e^{\left(  \frac{\left(  \gamma-1\right)  \left(2t-4\right)  }{4t}\right)  }}$, and
\vspace{.25cm}
\item[] $\displaystyle r^{\prime}(t)=\log_{2}e^{\left(  \frac{\left(  \gamma-1\right)  \left(2t-4\right)  }{4t}\right)}$.
\end{itemize}
\end{corollary}

\begin{figure}
\centering
\begin{tabular}{cc}
\includegraphics[width=0.4\textwidth]{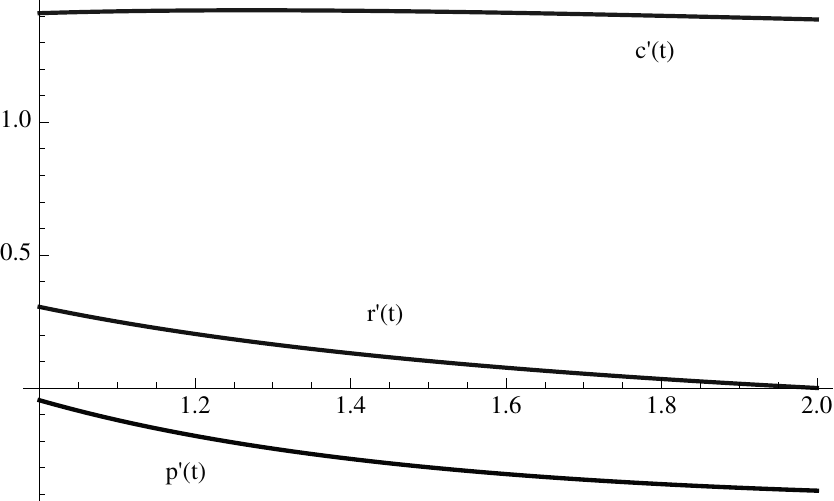} & \includegraphics[width=0.4\textwidth]{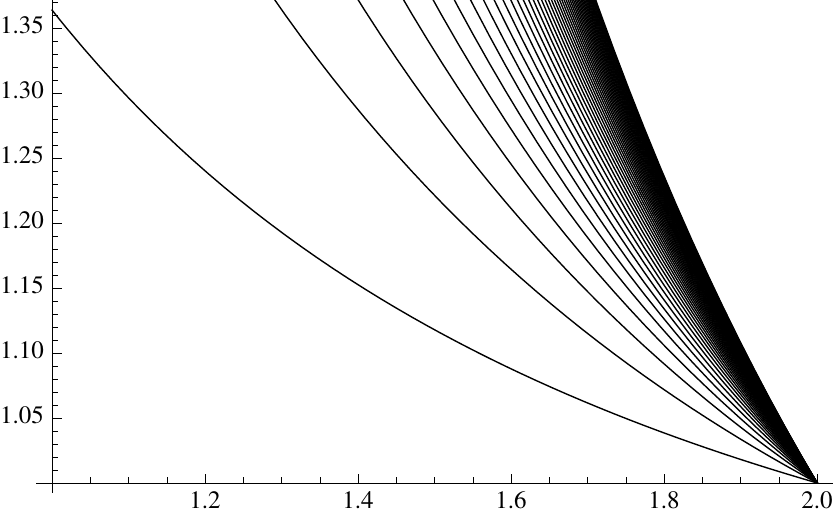}\\
(a) Plots of $p'(t), c'(t)$, and $r'(t)$ for $t \in [1,2)$. & (b) $K_{n,t}^{\mathbb{C}}$ for $2 \le n \le 50$.
\end{tabular}
\caption{}\label{pcrx}
\end{figure}

\medskip
\noindent {\bf Acknowledgements.} The authors thank Diogo Diniz for fruitful conversations on the topic of this paper.


\end{document}